\documentclass[10pt]{amsart}
\usepackage{fullpage, enumitem}
\usepackage{thmtools}
\usepackage{color}
\usepackage{accents}
\usepackage{refcount}

\newtheorem{theorem}{Theorem}[section]

\newtheorem{lemma}[theorem]{Lemma}
\newtheorem{proposition}[theorem]{Proposition}
\newtheorem{fact}[theorem]{Fact}
\newtheorem{claim}[theorem]{Claim}
\newtheorem*{claim*}{Claim}

\theoremstyle{definition}

\newtheorem{remark}[theorem]{Remark}

\newtheorem{observation}[theorem]{Observation}

\newcommand{\floor}[1]{\left\lfloor {#1} \right\rfloor}
\newcommand{\ceiling}[1]{\left\lceil {#1} \right\rceil}

\def\eps{\varepsilon}
\def\ex{\mathbb{E}}

\newcommand{\ep}{\varepsilon}

\newcommand{\MK}{\widetilde{K}}

\begin{document}

\title{Rainbow Cliques in Edge-Colored Graphs}

\author[A.~Czygrinow]{Andrzej Czygrinow}
\address{School of Mathematical and Statistical Sciences, 
Arizona State University, Tempe, AZ 85281, USA}
\email{aczygri@asu.edu}  

\author[T.~Molla]{Theodore Molla}\thanks{The second author was partially supported by 
NSF Grants DMS~1800761 and DMS~2154313.}  
\address{Department of Mathematics and Statistics, University of 
South Florida, Tampa, FL 33620, USA}
\email{molla@usf.edu}  

\author[B.~Nagle]{Brendan Nagle}
\address{Department of Mathematics and Statistics, University of 
South Florida, Tampa, FL 33620, USA}  
\email{bnagle@usf.edu}  

\date{\today}

\begin{abstract}
  Let $G = (V,E)$ be an $n$-vertex graph and let $c: E \to \mathbb{N}$ be a coloring of its edges.
  Let $d^c(v)$ be the number of distinct colors on the edges at $v \in V$ and let $\delta^c(G) = \min_{v \in V} \{ d^{c}(v) \}$.
  H.\ Li proved that $\delta^c(G) > n/2$ guarantees a rainbow triangle
  in $G$.  
  We give extensions of Li's result to cliques $K_r$ for $r \ge 4$.
\end{abstract}

\maketitle

\section{Introduction}
An {\it $n$-vertex edge-colored graph} $(G, c)$
consists of an $n$-vertex graph 
$G = (V, E)$ and an edge-coloring $c : E \to {\mathcal C}$ with some palette ${\mathcal C}$ of colors (usually ${\mathcal C} = \mathbb{N})$.   
A subset of edges 
$F \subseteq E$ is {\it rainbow} when $c|_F$ is injective and
a subgraph $H \subseteq G$ is \textit{rainbow} when $E(H)$ is rainbow.
The function $c$ is a \textit{proper coloring} of $G$ if 
the set $E_G(v, V)$
of edges incident to each $v \in V$ is rainbow.
The \textit{color degree} 
$d^c_G(v)$ of $v \in V$
is the number of colors appearing on $E_G(v, V)$, and    
$\delta^c(G) := \min \{d^c_G(v) : v \in V\}$.
H.~Li \cite{li2013rainbow} proved that $\delta^c(G) > n/2$ guarantees
a rainbow triangle in $G$, and B.~Li, Ning, Xu, \& Zhang \cite{li2014rainbow}
proved the following extension.   

\begin{theorem}[Li et al.~\cite{li2014rainbow}]\label{thm:liningxuzhang}
  For every $4 \neq n \in \mathbb{N}$, 
  every $n$-vertex edge-colored graph $(G,c)$ with  
  $\delta^c(G) \ge n/2$ and no rainbow triangle satisfies that 
$n$ is even, $G = K_{n/2, \, n/2}$, and $c$ is
  a proper edge-coloring of $G$.  
\end{theorem}  
\noindent  $\big($When $n = 4$, the same hypothesis allows for 
  improper colorings of $G \in \{K_4, K_4 - e\}$.$\big)$

The main result of this paper is an extension of Theorem~\ref{thm:liningxuzhang} for large graphs.
\begin{restatable}{theorem}{rainbowKrExact}\label{thm:rainbowKrExact}
  For every integer $s\ge 2$, every sufficiently large $n$-vertex edge-colored graph 
  $(G, c)$ with $\delta^c(G) \ge \left( 1 - \frac{1}{s-1} \right)n$
  and no rainbow $K_s$ satisfies that  
  $s-1$ divides $n$, $G$ is a complete balanced $(s-1)$-partite graph,
  and $c$ is a proper edge-coloring of $G$.
\end{restatable}

We also consider results for graphs with 
large {\it average color degree} $\ex[d^c_G(v)] = (1/|V|) \sum_{v \in V(G)} d^c_G(v)$.  

\begin{theorem}[Li at al~\cite{li2014rainbow}]\label{thm:liavgrainbow} 
  Every $n$-vertex edge-colored graph $(G,c)$ with $\ex[d^c_G(v)] \ge (n+1)/2$
contains a rainbow triangle $($which is best possible from $G = K_{[n]}$ and
$c(ij) = j$ for all integers $1 \leq i < j \leq n)$.
  \end{theorem}

\noindent We prove an asymptotic extension of Theorem~\ref{thm:liavgrainbow}
together 
with conclusions for some complete multipartite graphs.  For these,
let $K_s^\ell$ be the complete balanced $s$-partite graph with parts 
each 
of size $\ell$.
For vertex disjoint graphs $G$ and $H$, let $G \vee H$ be the join of $G$ and $H$, consisting of $G$, $H$, and 
$K\big[V(G), V(H)\big]$. 



\begin{restatable}{theorem}{rainbowKr}\label{thm:rainbowKr}
Fix $\gamma > 0$, $\ell, r, s \in \mathbb{Z}$ with $\ell \geq 1$, $r \geq 0$, and $s \geq \max\{1 + 2r, 2\}$,    
and a sufficiently large $n$-vertex edge-colored graph $(G, c)$.
  \begin{enumerate}
      \item[{\rm (i)}]  
      If $\ex[d^c_G(v)] \ge \big( 1 - \frac{1}{s-1} + \gamma \big)n$, then $(G, c)$
  contains a rainbow $K_{\floor{s/2}} \vee K^\ell_{\ceiling{s/2}} \supseteq K_s$.
\item[{\rm (ii)}]  If $\ex[d^c_G(v)] \ge 
  \big( 1 - \frac{1}{2(s - 1 - r)} + \gamma \big)n$,
  then $(G, c)$ contains a rainbow $K_r \vee K^\ell_{s - r}$ $(=K_s^{\ell}$ when $r=0)$.
  \end{enumerate}
  Moreover, these results are asymptotically sharp.
\end{restatable}

It is easy to 
demonstrate the asymptotic sharpness 
of Theorem~\ref{thm:rainbowKr}, and we do so in this introduction.  
However, 
we first call attention to  
the wide gap that Theorem~\ref{thm:rainbowKr} establishes between average color degree thresholds for rainbow $K_s$ and $K_s^{\ell}$ subgraphs.    
Even for minimum color degree, it is an open and perhaps interesting problem to determine
the threshold for a fixed $s$-chromatic rainbow subgraph $F$.

\subsection*{Proof of Theorem~\ref{thm:rainbowKr}:  asymptotic sharpness}  Statement~(i) of Theorem~\ref{thm:rainbowKr} is asymptotically sharp 
for $\ell \geq 1$ and $s \geq 2$
by
a rainbow $K_{s-1}^L$ for $n = L(s-1)$.  For Statement~(ii), we prepare the following considerations.  Orient each 
edge of a simple graph $G$
(in a unique direction) 
for an {\it oriented graph} $D$.  Ignore the orientations of arcs in an oriented graph $D$ for a simple graph $G = G(D)$.  Record these orientations
by the edge-coloring $c = c(D)$ given by $c(uv) = v$ when $(u, v) \in E(D)$.

\begin{observation}
\label{obs:oriented_coloring}
For an oriented graph $D$, 
no $F \subseteq G = G(D)$ with $|E(F)| > |V(F)|$ is properly colored by $c = c(D)$ because 
in $D$ 
the in-degree of 
some $w \in V(F)$ is at least $|E(F)| / |V(F)| > 1$.  Optimally, $D$ is a regular tournament where   
$\delta^c(G) = 1 + \delta^+(D) = (n+1)/2$ and $(G, c)$ has no rainbow $F$.  
\end{observation}

For Statement~(ii), 
fix $\ell \geq 1 + s - r$ for $s \geq \max \{1 + 2r, 2\}$
and set 
$n = L(s - 1 - r)$
for an odd $L \in \mathbb{N}$.  
For $G = K_{[n]}$, let $G_0 = K^L_{s-1-r}$ be one of its subgraphs with vertex partition $[n] = V_1 \, \dot\cup \, \dots \, \dot\cup \, V_{s-1-r}$
into classes $V_i$ of size $L$.  Let $D$ be a regular tournament placed on each of these classes $V_i$.  
Define $c: E(G) \to \mathbb{N}$ piecewise by letting $c|_{E(G_0)}$ map injectively into $\mathbb{N} \setminus [n]$
and by setting $c|_{E(G[V_i])} = c(D)$ for each class $V_i$.  Every $v \in [n]$ has common color-degree
$$
d_G^c(v) = n - L + 
\tfrac{L+1}{2} > \big(1 - \tfrac{1}{2(s - 1 - r)} \big) n.
$$
Moreover, $(G, c)$ admits no rainbow $K_r \vee K_{s-r}^{\ell}$, nor even a rainbow $K_{s-r}^{\ell}$.  
Indeed, let $G[U]$ be a copy of $K_{s-r}^{\ell}$ in $G$ with 
vertex partition $U = U_1 \, \dot\cup \, \dots \, \dot\cup \, U_{s-r}$ into classes $U_j$ of size $\ell$.  
Some class 
$V_i$ above satisfies 
$$
|V_i \cap U | \geq \tfrac{|U|}{s - 1 - r} = \tfrac{\ell(s-r)}{s-1-r} > \ell + 1, \quad \text{i.e.,} \quad 
|V_i \cap U | \geq \ell + 2
$$
because $\ell \geq 1 + s - r$.  Now, some class $U_j$ above similarly satisfies 
$$
\ell \geq |V_i \cap U_j| \geq \tfrac{|V_i \cap U|}{s - r} > \tfrac{\ell}{s - r} \geq 1, \quad \text{i.e.,} \quad 
\ell \geq |V_i \cap U_j| \geq 2. 
$$
Now, $F = G_0[V_i \cap U]$ has $\ell + k \geq \ell + 2$ vertices with $p \in [2, \ell]$ of them in $U_j$.  
But $|E(F)| \geq p (\ell + k - p) > \ell + k = |V(F)|$
easily holds for $p = 2$
from $k \geq 2$ and $\ell \geq 3$ and otherwise from 
$$
p + 1 + \tfrac{1}{p-1} < p + 2 \leq \ell + 2 \leq \ell + k \quad \implies \quad p^2 < (\ell + k)(p - 1)
\quad \implies \quad \ell + k < p (\ell + k - p).  
$$
Observation~\ref{obs:oriented_coloring} now guarantees that $F$ is not rainbow.

\subsection*{Itinerary of paper}
The main tools of this paper are 
regularity and stability 
methods.  
To facilitate 
these, we need some elementary ties 
among 
edge-colored graphs $(G, c)$, directed graphs $D$, and 
multigraphs $M$.  
Section~2 outlines (for pervasive future reference) these elementary details.
Section~3 presents a directed version of the Szemer\'edi Regularity Lemma~\cite{SRL}  
due to Alon and Shapira~\cite{alon2003testing}.  
Section~3 also establishes a compatible `rainbow embedding lemma' needed in this paper.  
These tools quickly give our 
proof of 
Theorem~\ref{thm:rainbowKr}, 
which we present in Section~4.  
Section~5 presents our main appeal to 
the Stability Method in the form of an `extremal lemma' and a `non-extremal lemma'.  
These tools quickly give our proof of 
Theorem~\ref{thm:rainbowKrExact}, which we also present in Section~5.  
The remainder of the paper proves these lemmas.  
Some of 
this work relies on  
the earlier 
regularity tools, but most 
depends on 
somewhat technical adhoc arguments.

\section{Elementary details:  Edge-colored Graphs, Directed Graphs, and Multigraphs}  
\label{sec:2}
Our paper frequently appeals to 
elementary ties 
among 
edge-colored graphs $(G, c)$, directed graphs $D$, and multigraphs $M$.  
It also frequently appeals to elementary observations (easy facts) about these objects.  
The current section centralizes a 
repository, aiding future reference, 
of these connections and observations.




\subsection*{Edge-colored graphs} 
Recall that an edge-colored graph $(G, c)$ is a simple graph $G = (V, E)$ 
together with an edge-coloring $c: E \to \mathcal{C}$ to some palette $\mathcal{C}$ of colors.  
Sometimes, we prefer for $(G, c)$ to be {\bf edge-minimal}, meaning that it 
contains no monochromatic path on three edges.
However, 
every edge-colored graph $(G, c)$ admits
a spanning subgraph 
$F \subseteq G$ for which 
$\big(F, c|_{E(F)}\big)
= \big(F, c_F\big)$
is edge-minimal
and 
{\bf color-degree preserving}: $d_F^{c_F}(v) = d_G^c(v)$ for all 
$v \in V$.  
Indeed, iteratively delete $uv \in E(G)$ when $c(uv)$ appears at least twice on each of $E_G(u, V)$ and $E_G(v, V)$.  


\subsection*{Simple directed graphs}
An oriented graph $D$ 
is a directed graph in which every $u \neq v \in V(D)$ yield 
at most one copy of 
either 
$(u, v)$ or $(v, u)$ in $E(D)$.
Recall that an oriented graph $D$ 
yields an edge-colored graph $(G, c)$ by 
$c(uv) = v$ for $(u, v) \in E(D)$.  
We next convert edge-colored (not-necessarily edge-minimal) 
graphs $(G, c)$ to directed (but not-necessarily oriented) graphs $D$.  
Indeed, 
for a real $m \geq 1$, define  
a {\bf $\boldsymbol{(G, c, m)}$-digraph} $D$ to be any directed graph obtained in the following way:
for each $v \in V(G)$ and $q \in {\rm im}(c)$, whenever   
$Q = \{w \in N_G(v): \, c(vw) = q\}$ satisfies $1 \leq |Q| \leq m$, 
arbitrarily select a unique $w = w(v, q) \in Q$ and place $(v, w) \in E(D)$.
Note that 
$(G, c, m)$-digraphs are {\bf simple}:  
every $(v, w) \in V \times V$ 
appears at most once in $E(D)$, and only when $v \neq w$.  
{\sl In fact, all directed graphs $D$ of this paper are (tacitly) simple.}

Define the {\bf complement} $\overline{D}$ of a 
directed graph $D$ to be the simple directed graph given by
$(u, v) \in E(\overline{D})$ if, and only if, $(u, v) \not\in E(D)$ for all $u \neq v \in V(D)$.  
When both $(v, w), (w, v) \in E(D)$, we say
that 
$\{v, w\}$ induces a {\bf $2$-cycle} in $D$.  We define 
$H = H(D)$ to be the simple 
{\bf $2$-cycle graph} on $V(D)$ whose edges $vw$ are the 2-cycles $\{v, w\}$ of $D$.  

\subsection*{$\boldsymbol{2}$-cycle `near' cliques}
The following peculiar directed graphs 
$D$
appear pervasively in this paper. 
We define these in terms of their complements $\overline{D}$, which are $[s]$-vertex oriented graphs 
where $E(\overline{D})$ consists of 
\begin{enumerate}
    \item[{\rm (i)}]  some subset 
    $\triangle_{\%}$ 
    of exactly one cyclically oriented triangle $(i, j)$, $(j, k)$, $(k, i) \in E(D)$, and 
    \item[{\rm (ii)}]  an oriented matching ${\mathcal M}_r$ of 
    some 
    $0 \leq r \leq (s-3)/2$
    many 
    pairs $(g, h)$ disjoint from $\{i, j, k\}$.
\end{enumerate}
We write $D = \mathbb{K}_s - \triangle_{\%} - {\mathcal M}_r$ for 
the complement of $\overline{D}$, where 
$\mathbb{K}_s$
denotes the set of all non-diagonal $(u, v) \in [s] \times [s]$.
When $\overline{D}$ consists only of some
oriented matching ${\mathcal M}_r$ of size $0 \leq r \leq s/2$, we write 
$D = \mathbb{K}_s - {\mathcal M}_r$.  
(The classes 
$\mathbb{K}_s - \triangle_{\%} - {\mathcal M}_r$ 
and 
$\mathbb{K}_s - {\mathcal M}_r$ clearly overlap.)


\subsection*{Multigraphs}
All multigraphs $M$ in this paper are {\bf standard}:  they are loop-free
with edge multiplicity $\mu_M(e) \leq 2$ for all $e \in \binom{V}{2}$, where $V = V(M)$.
We say that 
an edge $e \in E(M)$ is {\bf heavy} 
when $\mu_M(e) = 2$ and is 
{\bf light} when 
$\mu_M(e) = 1$.  
Note that a simple directed graph $D$ reduces to a standard multigraph $M(D)$ 
by removing the orientations of arcs $(u, v) \in E(D)$.
Thus, the simple graph $G(D)$ 
underlying $D$ is the same as the simple graph $G(M(D))$ underlying $M(D)$.  Moreover, 
the 2-cycle graph $H(D)$ is the same as the simple graph 
$H(M(D))$ of heavy edges of $M(D)$.  

Continuing, 
define the {\bf complement} $\overline{M}$ of a standard multigraph $M$ by 
$\mu_{\overline{M}}(e) = 2 - \mu_M(e)$ for all $e \in \binom{V}{2}$.
Note that 
$e(M) = \sum_{e \in \binom{V}{2}} \mu_M(e)$
counts the edges of $M$.
Fix 
$v \in V$ and $U, U' \subseteq V$. Define 
\begin{multline*}
d_M (v, U) = \sum_{u \in U} \mu_M(uv), \qquad 
d_M(v) = d_M(v, V), 
\qquad 
\delta(M) = \min \{d_M(v): \, v \in V\},  \\
\Delta(M) = \max \{d_M(v): \, v \in V\}, \qquad \text{and} \qquad   
e_M(U, U') = \sum_{u \in U} d_M(u, U')
= \sum_{u' \in U'} d_M(u', U).
\end{multline*}
Note that 
$e_M(U, U') = 
e_M(U, U' \setminus U) 
+ 
2 e_M(U \cap U')$ also holds.



Let 
an $s$-set 
$S \subseteq V$ 
induce in $\overline{M}$ an $r$-matching of light edges.  Here, 
we abuse notation from 
directed graphs and write $M[S] = \mathbb{K}_s - {\mathcal M}_r$.  
The following (non-elementary) theorem on these structures 
essentially appeared as Theorem~4.2 of~\cite{CDMT} 
and is used pervasively throughout the paper.  

\begin{theorem}[\cite{CDMT}]
    \label{lem:multigraphturan}
Fix an integer $s \geq 2$
and an $n$-vertex standard multigraph $M$.  
\begin{enumerate}
    \item[{\rm (1)}]  
If 
$e(M) > \big(1 - \tfrac{1}{s-1}\big)n^2$, then $M$
contains $\mathbb{K}_s - {\mathcal M}_q$ for some integer $0 \leq q \leq s/2$.   
\item[{\rm (2)}]  
If $e(M) 
> \big(1 -\tfrac{1}{2(s-1 - r)}\big)n^2$
for an integer $0 \leq r \leq (s-1)/2$, 
then $M$    
contains $\mathbb{K}_s - {\mathcal M}_q$ for some integer $0 \leq q \leq r$.   
\end{enumerate}
\end{theorem}

\begin{remark}
\rm 
    \label{rem:CDMT}
    Statement~(2) 
of Theorem~\ref{lem:multigraphturan} did not 
explicitly appear in~\cite{CDMT} but it follows trivially from Statement~(1) which did.  
Indeed, 
Statement~(1) guarantees $T \subseteq V$ 
where $|T| = 2s - 2r - 1$ and $\overline{M}[T] = {\mathcal M}_q$ is a $q$-matching of light edges.  
When $q > r$, 
delete an endpoint from any $q - r$ edges from ${\mathcal M}_q$ to obtain $S \subseteq T$
where $|S| = 2s - 2r - 1 - (q - r) \geq s$ and $\overline{M}[S] = {\mathcal M}_r$
is an $r$-matching of light edges.  
\end{remark}

\subsection{Observations on edge-colored graphs}
In this subsection, 
$(G, c)$ is an edge-colored graph.  

\begin{observation}
    \label{obs:iterative_rainbow}
{\it 
Let $A \, \dot\cup \, B \subseteq V(G)$ satisfy 
$|B| > |A| \cdot e(G[A])$ and that $G[A, B]$ is properly colored by $c$.  Then some $b_0 \in B$ satisfies $c(ab_0) \not\in c(E(G[A])$ for all 
$a \in A$.  
}  
\end{observation}

\begin{proof}  
Indeed, for $C = c(E(G[A]))$ and $a \in A$,  
define
$E_C(a, B) = \{ ab \in E_G(a, B): \, c(ab) \in C \}$ 
and 
$N_C(a, B) = \{ b \in B: \, ab \in E_C(a, B) \}$.
Since $E_C(a, B)$ is rainbow, $|E_C(a, B)| = |N_C(a, B)| \leq |C|$ so 
$$
\Big| \bigcup_{a \in A} N_C(a, B) \Big| \leq \sum_{a \in A} \big|N_C(a, B)\big| \leq |A| |C| \leq |A| \cdot e(G[A]) < |B|  
$$
and the existence of $b_0 \in B$ follows. 
\end{proof}

Iterative applications of Observation~\ref{obs:iterative_rainbow} yield the following one.

\begin{observation}
    \label{obs:proper_to_rainbow2}
{\it 
Let $A \, \dot\cup \, B \subseteq V(G)$ satisfy the following conditions:
\begin{enumerate}
\item[(i)]
$K_a^{\ell}$ is a rainbow subgraph of $G[A]$; 
\item[(ii)] 
$K_b^L$ is a subgraph of $G[B]$ which is properly colored by $c$; 
\item[(iii)]
$G[A, B] = K[A, B]$ is properly colored by $c$.
\end{enumerate}
Then $(G, c)$ admits a rainbow $K_a^{\ell} \vee K_b^k$
whenever $L \geq L_0(k, \ell, a, b)$ is sufficiently large.  
}  
\end{observation}


\subsection{Observations on $\boldsymbol{(G, c, m)}$-digraphs}  
In this subsection, $(G, c)$ is an $n$-vertex edge-colored graph, $D$ is a $(G, c, m)$-digraph thereof (with $m \geq 1$), and $H = H(D)$ is the 2-cycle
graph of $D$.  

\begin{observation}
    \label{obs:GemD}
    {\it 
When $(G, c)$ is edge-minimal, every $uv \in E(G)$
yields $(u, v) \in E(D)$ or $(v, u) \in E(D)$.  
}  
\end{observation}

\begin{proof}
Indeed, $c(uv)$ is unique to $u$ or to $v$.  
\end{proof}

\begin{observation}
\label{obs:GtoDoutdeg}
{\it 
Every $v \in V(G)$ satisfies $d^+_D(v) \ge d^c_G(v) - \lfloor d_G(v) / (m+1) \rfloor$.  
Equality holds when $m = n - 1$.  
}  
\end{observation}

\begin{proof}
Indeed, 
$d_G(v)$ is $(m+1)$-fold  
in $d_G^c(v) - d_D^+(v)$ many colors, so 
$d_G(v) \geq (m + 1) 
\big(
d_G^c(v) - d_D^+(v) \big)$.
\end{proof}


\begin{observation}\label{obs:prop_coloring}
{\it 
The restriction 
$c|_{E(H)}$ 
is proper.   
}
\end{observation}

\begin{proof}  
Indeed, 
$E_G\big(v, N_D^+(v)\big)$ is rainbow for every $v \in V(G)$.   
\end{proof}  

\begin{observation}\label{obs:H_edges}
$e(H) + e(G) \geq \sum_{v \in V(G)} d^+_D(v)$.
{\it Equality holds when $(G, c)$ is edge-minimal  
$($cf.~Observation~\ref{obs:GemD}$)$.  }  
\end{observation}

\begin{proof} 
Indeed, $(v, w) \in E(D)$ yields $vw \in E(G)$ 
and 
$(v, w), (w, v) \in E(D)$ yields 
$vw \in E(H)$.
\end{proof}



\subsection{Observations on standard multigraphs}
In this subsection, $M$ is an $n$-vertex standard multigraph and $H = H(M)$ is its simple graph of heavy edges.  

\begin{observation}
\label{obs:mindeg}
{\it 
Fix $0 < \alpha < \beta^2 < \beta < 1$ and a sufficiently large integer $n \geq n_0(\alpha, \beta)$.  
Let $M$ 
satisfy that $e(M) \geq (d - \alpha)n^2$
and that 
every induced $m$-vertex $M' \subseteq M$
with $m \geq (1 - \beta)n$ has 
$e(M') \leq d m^2$.  
Then 
some such $M_0'$ 
also satisfies 
$\delta(M'_0) \geq 2 (d - \beta)m$.
}
\end{observation}  

\begin{proof}  
Indeed, set $M_0 = M$ and 
$M_t = M_{t-1} - v_t$
for 
some integer 
$t = \rho n$ 
with $\rho \in [0, 1]$
and some vertex $v_t \in V(M_{t-1})$ with $d_{M_{t-1}}(v) < 2(d - \beta) |V(M_{t-1})|$.  
 Assume, for a contradiction, 
that 
$\rho \geq \beta$
is possible in this context and take 
$t = \lceil \beta n \rceil$ precisely.
Then 
$$
(d - \alpha) n^2 - 2 t (d - \beta) n  <  
e(M_t) 
\leq 
d 
(n - t)^2      
$$
 gives 
$\rho^2 - 2\beta \rho + \alpha > 0$, so 
the quadratic formula yields
$\rho > \beta + \sqrt{\beta^2 - \alpha} 
> 
(1/n) \lceil \beta n \rceil = \rho$.
\end{proof}

\begin{observation}
\label{obs:extmultiHdeg}
{\it  
Let 
$\delta(M) \ge 2 \big(1 - \tfrac{1}{s-1} - \alpha \big)n$ and let 
$U \subseteq V(M)$ satisfy at least one of the properties:  
\begin{enumerate}
    \item[{\rm ($P_1$)}] 
{\it $U$ is independent and $|U| \geq \big(\tfrac{1}{s-1} - \beta\big)n$}; 
\item[{\rm ($P_2$)}]  
{\it $U$ has no three points spanning five or more edges of $M$ and 
$|U| \geq \big(\tfrac{2}{s-1} - \beta\big)n$.}     
\end{enumerate}  
Then 
$d_H\big(u, \overline{U}\big) \ge \big|\overline{U}\big| - (4\alpha  + 2 \beta) n$ for all $u \in U$.}
\end{observation}

\begin{proof}  
Indeed, 
fix $u \in U$ and consider the cases $(P_1)$ and $(P_2)$.    
\medskip 

\noindent {\sc Case 1 $\big(${\rm $U$ satisfies $(P_1)$}$\big)$.}
Let $L = M \setminus H$ be the 
(simple) graph of light edges.  Then
\begin{multline*}
d_H\big(u, \overline{U}\big) 
= 
d_M\big(u, \overline{U}\big) -  
d_L\big(u, \overline{U}\big)   
\stackrel{(P_1)}{\geq}  
\delta(M) - \big|\overline{U}\big|
= 
\delta(M) - n + |U|
\geq 
2 \big(1 - \tfrac{1}{s-1} - \alpha \big)n
- n + |U| \\
\stackrel{(P_1)}{\geq}  
\big(1 - \tfrac{1}{s-1} \big)n - 2\alpha n - \beta n 
\stackrel{(P_1)}{\geq}  
n - |U| - 2 (\alpha + \beta)  n
= 
\big| \overline{U} \big| - 2 (\alpha + \beta) n.  
\end{multline*}

\noindent {\sc Case 2 $\big(${\rm $U$ satisfies $(P_2)$}$\big)$.}
Either $d_M(u, U) = d_L(u, U) \leq |U| - 1$ or $uv \in E(H[U])$ gives
\begin{multline}
\label{eqn:4.29.2024.3:17p}
2|U| 
\stackrel{(P_2)}{\geq} \sum_{w \in U} 
\big(\mu_M(vw) + \mu_M(uw) \big)
= 
d_M(v, U) + d_M(u, U)
\geq 
\delta(M) - 2 \big|\overline{U}\big| + d_M(u, U) \\
\implies \qquad 
d_M(u, U) \leq 2n - \delta(M) \leq \big(\tfrac{2}{s-1} + 2\alpha\big) 
\stackrel{(P_2)}{\leq}
|U| + (2\alpha + \beta)n, 
\end{multline}  
so 
\begin{multline*}  
d_H\big(u, \overline{U}\big) \geq \delta(M) - d_L\big(u, \overline{U}\big) - d_M(u, U) 
\stackrel{\eqref{eqn:4.29.2024.3:17p}}{\geq}  
\delta(M) - \big|\overline{U}\big| - |U| - (2\alpha + \beta)n  =  
\delta(M) - n - (2\alpha + \beta)n  \\
\geq 
n - \big(\tfrac{2}{s-1} + 4\alpha + \beta \big)n 
\stackrel{(P_2)}{\geq}
n - |U| - (4 \alpha + 2\beta)n = 
\big|\overline{U}\big| - (4 \alpha + 2 \beta)n.  
\end{multline*}  
\end{proof}

\begin{observation}
\label{obs:deg}
{\it  
Fix integers 
$q \geq p \geq 2$,   
a 
$\beta \in (0, 1)$, 
and 
a sufficiently small 
$\alpha > 0$.   
Let 
$\delta(M) \ge  2 \big(\tfrac{q-2}{q-1} - \alpha\big)n$.  
Then 
all $U \subseteq V(M)$ with $|U| \ge \big(\tfrac{p-1}{q-1} - \alpha\big)n$ and $v \in V$ 
satisfy 
$d_M(v, U) \geq 
2\big(\tfrac{p-2}{p-1} - \beta) |U|$.  
}
\end{observation}

\begin{proof}  
Indeed, $d_M(v, U) + d_{\overline{M}}(v, U)$ is either $2|U|$ or 
$2(|U| - 1)$, so 
$$
d_M(v, U) \geq 2 |U| - 2 - \Delta\big(\overline{M}\big)
= 2|U| - 2n + \delta(M) 
\geq 2 \Big(1 - \big(\tfrac{1}{q-1} + \alpha\big) \tfrac{n}{|U|} \Big)|U| 
\geq 2 \Big(1 - \tfrac{1 + \alpha (q - 1)}{p - 1 - \alpha (q-1)} \Big)|U|, 
$$
which is 
$2\big(\tfrac{p-2}{p-1} - \beta) |U|$ when $\alpha = \alpha(\beta, q)$ is sufficiently small.  
  \end{proof}

\section{Digraph Regularity and Rainbow Embedding}
\label{sec:SRL}
Digraph regularity plays a crucial role in our proof of 
Theorem~\ref{thm:rainbowKr}.
First, 
we prepare 
upcoming 
Theorem~\ref{thm:rainbowKr}, 
which 
is 
a directed version 
of the Szemer\'edi Regularity Lemma~\cite{SRL}  
due to Alon and Shapira~\cite{alon2003testing}.  
Second, we develop and prove 
upcoming Lemma~\ref{lem:MDr_to_rainbowKr}, 
which is a compatible `rainbow clique embedding lemma' 
(which involves 
the directed structures 
$\mathbb{K}_s - \triangle_{\%} - \mathcal{M}_r$
and 
$\mathbb{K}_s - \mathcal{M}_r$).

\subsection{A directed regularity lemma}
\label{sec:DRL}
We begin with some routine definitions.  
Fix $\eps > 0$, a graph $G = (V, E)$, and disjoint $\emptyset \neq A, B \subset V$.  
Define the \emph{density of $G$ w.r.t.~$(A, B)$} by 
    $d_G(A, B) = e(G[A, B]) / (|A||B|)$. 
    For $d \geq 0$, 
we say that $G[A, B]$
is {\it $(d, \eps)$-regular} when 
$\big| d_G(X, Y) - d \big| < \eps$ 
  for all $(A', B') \in 2^A \times 2^B$ satisfying $|A'| > \eps |A|$ and $|B'| > \eps |B|$.
  We say that $G[A, B]$ is {\it $\eps$-regular} when it is $(d, \eps)$-regular for some $d \geq 0$.
For $d_0 \geq 0$ fixed, 
we say that 
  $G[A, B]$ is $(\geq\! d_0, \eps)$-\emph{regular} when $G[A, B]$ 
  is $(d, \eps)$-regular for some $d \geq d_0$. 



Szemer\'edi's regularity lemma states that for all 
$\eps > 0$, every graph $G = (V, E)$ admits a partition 
$V = V_0 \, \dot\cup \, V_1 \, \dot\cup \, \dots \, \dot\cup \, 
V_t$ for $t \leq T_0(\eps)$
where $|V_1| = \dots = |V_t| > (1 - \eps) |V| / t$ and 
all but $\eps t^2$ pairs $1 \leq i < j \leq t$ satisfy that
$(V_i, V_j)$ is $\eps$-regular.  We need a version
of this lemma for a directed graph $D$.  To that end, fix disjoint $\emptyset \neq A, B \subset V(D)$.  We 
write $D\langle A, B\rangle$ for the simple bipartite  graph 
with vertex classes $A$ and $B$ and all edges $ab$ of the form 
$(a, b) \in E(D)$.  Note here that the order $(A, B)$
is important as often $D \langle A, B \rangle \neq D \langle B, A \rangle$.  
We then write 
$d_D(A, B) = e(D \langle A, B \rangle ) / (|A| |B|)$ as a {\it density} of $D$ w.r.t.~$(A, B)$.   
We also write $H[A, B] = D \langle A, B \rangle \cap D \langle B, A \rangle$ for the graph of pairs  
$ab \in E(H[A, B])$ satisfying $(a, b), (b, a) \in E(D)$.  


Alon and Shapira~\cite{alon2003testing} proved 
a directed version of 
Szemer\'edi's regularity lemma, whereby for all $\eps > 0$, every digraph $D$  
admits a partition 
$V(D) = V_0 \, \dot\cup \, V_1 \, \dot\cup \, \dots \, \dot\cup \, V_t$ 
for $t \leq T_0(\eps)$
where $|V_1| = \dots = |V_t| \geq (1 - \eps) |V| / t$
and all but $\eps t^2$ pairs $1 \leq i < j \leq t$
satisfy that each of $D \langle V_i, V_j \rangle$, $D\langle V_j, V_i \rangle$, and $H[V_i, V_j]$ is $\eps$-regular.
We need the following so-called `degree form' of~\cite{alon2003testing}, which essentially appears as Lemma~39 
of Taylor~\cite{Taylor}

  
\begin{theorem}
\label{thm:degree_form}
For all $\eps > 0$ and $t_0 \in \mathbb{N}$, there exist $T_0 = T_0(\eps, t_0) \in \mathbb{N}$ and $N_0 = N_0 (\eps, t_0) \in \mathbb{N}$
so that for all $n \geq N_0$ and $d_0 \in [0, 1)$, every $n$-vertex digraph $D$
admits a partition 
$V(D) = V_0 \, \dot\cup \, V_1 \, \dot\cup \, \dots \, \dot\cup \, V_t$ 
satisfying 
$t_0 \leq t \leq T_0$ and $|V_1| = \dots = |V_t| \geq (1 - \eps) n / t$ 
and a spanning subdigraph $D'$ of $D$ satisfying 
\begin{enumerate}
    \item[{\rm (i)}]  
    $d^+_D(v) - d^+_{D'}(v), \, 
    d^-_D(v) - d^-_{D'}(v) 
    \leq 
    (3d_0+\ep)n$  
    for all $v \in V(D)$; 
    \item[{\rm (ii)}]  
    $E(D'[V_1]) = \dots = E(D'[V_t]) = \emptyset$; 
    \item[{\rm (iii)}]  
$E(D'\langle V_i, V_j \rangle) \neq \emptyset$ $\implies$ $D\langle V_i, V_j \rangle$ is $(\ge\!d_0, \eps)$-regular; 
    \item[{\rm (iv)}]  
$E(D' \langle V_i, V_j \rangle) \cap  
E(D'\langle V_j, V_i \rangle) \neq \emptyset$ 
$\implies$ each of $D \langle V_i, V_j \rangle$, $D \langle V_j, V_i \rangle$, and $H[V_i, V_j]$ is $(\ge\!d_0, \eps)$-regular.   
\end{enumerate}
\end{theorem}


\subsection{A rainbow clique embedding lemma}
\label{sec:REL}
The main goal of this section is to prove the following `rainbow clique embedding lemma'.  

\begin{lemma}[Rainbow embedding lemma] 
\label{lem:MDr_to_rainbowKr}
For all $\ell, s \in \mathbb{N}$
and $d_0 \in (0, 1)$, there exist $\eps, \tau \in (0, 1)$ so that whenever 
an edge-colored $V$-vertex graph $(F, c)$ and $[s]$-vertex directed graph $D_F$ 
satisfy 
that 
\begin{enumerate}
\item[{\rm (a)}]  $F$ is $s$-partite with $s$-partition 
$V = V_1 \, \dot\cup \, \dots \, \dot\cup \, V_s$, where 
$|V_1| = \dots = |V_s| = m$ is sufficiently large,   
\item[{\rm (b)}]  $F[V_i, V_j]$ is $(\geq\! d_0, \varepsilon)$-regular for each $1 \leq i < j \leq s$, and    
\item[{\rm (c)}]  
all $(i, j) \in D_F$ and $v_i \in V_i$ satisfy that $E_F(v_i, V_j)$ is rainbow and every  
$\phi \in c\big(E_F(v_i, V_j)\big)$ appears at most $\tau m$ times on $E_F(v_i, V \setminus V_j)$,       
\end{enumerate}
then the following properties hold:  
\begin{enumerate}
\item[{\rm (i)}]
if $D_F = \mathbb{K}_s - \triangle_{\%} - {\mathcal M}_r$,  then $(F, c)$ admits a rainbow $K_s$;     
\item[{\rm (ii)}]
if $D_F = \mathbb{K}_s - {\mathcal M}_r$, then $(F, c)$ admits a rainbow $K_r \vee K_{s - r}^{\ell}$.    
\end{enumerate}  
\end{lemma}

\noindent  It is easy to jointly apply 
Theorem~\ref{thm:degree_form}   
and Lemma~\ref{lem:MDr_to_rainbowKr}.  For future reference, we sketch how this is done.

\begin{remark}
\label{rem:embedding}
\rm 
Fix $\ell, s \in \mathbb{N}$ and $d_0 \in (0, 1)$ and let $\eps, \tau \in (0, 1)$ be given by Lemma~\ref{lem:MDr_to_rainbowKr}.  
Let $(G, c)$ be a large $n$-vertex graph, and     
construct a $(G, c, \sqrt{n})$-digraph $D$ thereof.
For $d_0, \eps$ above and $t_0 = 1$, 
let $V(D) = V_0 \, \dot\cup \, V_1 \, \dot\cup \, \dots 
\, \dot\cup \, V_t$ 
and $D' \subseteq D$ be those objects guaranteed by Theorem~\ref{thm:degree_form}.  
Assume that 
\begin{equation}
\label{eqn:embedding}
\text{{\it $D' - V_0$ contains a copy 
$\hat{K}$
of 
$\mathbb{K}_s - \triangle_{\%} - {\mathcal M}_r$
or 
of 
$\mathbb{K}_s - {\mathcal M}_r$.}}
\end{equation}
Every $u \neq v \in V(\hat{K})$ 
spans an arc of $\hat{K}$ while $E(D'[V_1]) = \dots = E(D'[V_t]) = \emptyset$, so 
$V(\hat{K})$ traverses $V_1 \, \dot\cup \, \dots \, \dot\cup \, V_t$.  Let, w.l.o.g., 
$V(\hat{K})$ be ordered as $(v_1, \dots, v_s) \in V_1 \times \dots \times V_s$.  
We now define an $s$-partite graph $F$ with $s$-partition $V = V_1 \, \dot\cup \, \dots \, \dot\cup \, V_s$ and corresponding 
$[s]$-vertex directed graph
$D_F$ as follows.  
First, 
for each $1 \leq i \neq j \leq s$, put $(i, j) \in E(D_F)$ if, and only if, $(v_i, v_j) \in E(\hat{K})$, noting that 
at least one of $(i, j)$ and $(j, i)$ must belong to $E(D_F)$.  
Second, define 
$$
F[V_i, V_j]
= 
\left\{
\begin{array}{cc}
D \langle V_i, V_j\rangle & \text{if $(i, j) \in E(D_F)$ but $(j, i) \not\in E(D_F)$,} \\
H[V_i, V_j] & \text{if both $(i, j), (j, i) \in E(D_F)$.}
\end{array}
\right.
$$
Theorem~\ref{thm:degree_form} guarantees that $F[V_i, V_j]$ is $(\ge\! d_0, \eps)$-regular
and that 
$m = |V_1| = \dots = |V_t|\geq (1 - \eps) n / t \geq n / (2T_0) = \Omega(n)$ is large, where $T_0 = T_0(\eps, t_0 = 1)$ is the constant
guaranteed by 
Theorem~\ref{thm:degree_form}.
The construction of $D_F$ 
guarantees that $E_F(v, V_j)$ is rainbow for each $v \in V_i$ and that, moreover, 
every $\phi \in c(E_F(v, V_j))$ appears at most $\sqrt{n} \ll \tau n$ times in $E_F(v, V \setminus V_j)$.  
Now, depending on which of~\eqref{eqn:embedding}
transpired, 
Lemma~\ref{lem:MDr_to_rainbowKr} guarantees that $(F, c|_F)$ and hence $(G, c)$ admits a rainbow $K_s$ or a rainbow
$K_r \vee K_{s - r}^{\ell}$.  \hfill $\Box$
\end{remark}


\subsection{Proof of Lemma~\ref{lem:MDr_to_rainbowKr}}
Fix $\ell, s \in \mathbb{N}$ and $d_0 \in (0, 1)$.  
We take $\eps, \tau > 0$ sufficiently small 
and $m \in \mathbb{N}$
sufficiently large 
whenever needed.  
Let
$(F, c)$ and $D_F$ satisfy 
Hypotheses~(a)--(c) of Lemma~\ref{lem:MDr_to_rainbowKr}. 

\subsection*{Conclusion~(i)}
Uniformly at random select
$(v_1, \dots, v_s) \in V_1 \times \dots \times V_s$.  
We prove
\begin{equation}
\label{eqn:RCLi}
\sum 
\Big\{ 
{\mathbb P} \big[ c(\{v_i, v_j\} = c(\{v_{i'}, v_{j'}\}) \big] : \, 
\{i, j\} \neq \{i', j'\} \in \tbinom{[s]}{2}
\Big\}
\leq  
\tbinom{s}{2}^2 \big(\tau + (1/m) \big) \leq 2\tau s^4
\end{equation}
is small.  
(Throughout this section, we use that ${\mathbb P}\big[c(\{v_i, v_j\}) = c(\{v_{i'}, v_{j'}\})\big] = 0$ when $\{v_i, v_j\} \not\in E(F)$ or $\{v_{i'}, v_{j'}\} \not\in E(F)$
since then 
$c(\{v_i, v_j\}) = c(\{v_{i'},v_{j'}\})$
is the empty event.)
The well-known 
Graph Counting Lemma 
says that (with $o_{\eps}(1) \to 0$ as $\eps \to 0$)     
\begin{equation}
\label{eqn:GCLi}
\mathbb{P} \Big[ F\big[\{v_1, \dots, v_s\}\big] = K_s\Big] = 
\big(1 \pm o_{\eps}(1) \big)
\prod_{1 \leq i < j \leq s} d_F(V_i, V_j) 
\ge 
\big(1 - o_{\eps}(1) \big)
d_0^{\binom{s}{2}} 
\end{equation}
is big.
Then~\eqref{eqn:RCLi} 
and~\eqref{eqn:GCLi}
guarantee 
that most $K_s$ in $F$ are rainbow.  
To prove~\eqref{eqn:RCLi},   
fix $\{i, j\} \neq \{i', j'\} \in \binom{[s]}{2}$.     
W.l.o.g., 
$(i, j), (i', j') \in E(D_F) =  
E\big( \mathbb{K}_s - \triangle_{\%} - {\mathcal M}_r\big)$.   
\bigskip 

\noindent {\sc Case 1 ($\{i , j\} \cap \{i', j' \} = \emptyset$).}
Since $E_F(v_i, V_j)$ is rainbow, 
${\mathbb P} \big[\, c (\{v_i, v_j\}) = c(\{v_{i'}, v_{j'}\}) \, \big] 
\leq 
1 / m$.  
\bigskip

\noindent {\sc Case 2 ($j = i'$, $j' = k$).}  
Since $c(\{v_j, v_k\})$ 
sees $E_G(v_j, V_i)$ 
at most 
$\tau m$ times, 
${\mathbb P} \big[\, c (\{v_i, v_j\}) = c(\{v_j, v_k\}) \, \big] 
\leq \tau.$

\subsection*{Conclusion~(ii)}
W.l.o.g., let 
\begin{equation}
\label{eqn:4.17.2024.6:11p}
E(\overline{D}_F) = 
{\mathcal M}_r = 
\big\{ (2, 1), (4, 3), \dots, (2r,\, 2r-1) \big\}.  \end{equation}
For each $1 \leq i \leq s$, uniformly at random select $U_i \in \binom{V_i}{\lambda_i}$ where 
$$
\lambda_i = 
\left\{
\begin{array}{cc}
1 & \text{if $i \in \{1, 3, \dots, 2r - 1\}$,} \\
\ell & \text{else.}
\end{array}  
\right.
$$
Let $U = U_1 \, \dot\cup \, \dots \, \dot\cup \, U_s$
and 
$\binom{U}{2}_{\times} = 
\binom{U}{2} \setminus \big(\binom{V_1}{2} \, \dot\cup \, \dots \, \dot\cup \, \binom{V_s}{2} \big)$.      
We prove
that 
\begin{equation}
    \label{eqn:RCLii}
\sum 
\Big\{
\mathbb{P} \big[
c(\{u_i, u_j\}) = c(\{u_{i'}, u_{j'}\}) \big]: \, 
\{u_i, u_j\} \neq \{u_{i'}, u_{j'}\} \in \tbinom{U}{2}_{\times} 
\Big\}  
\leq 
\Big(\tbinom{s}{2} \ell^2\Big)^2\big(\tau \ell + (2 \ell / m)  \big) \leq 2 \tau s^4 \ell^5  
\end{equation}  
is small. 
The Graph Counting Lemma says 
that 
$$
{\mathbb P}  
\big[F[U] 
= K_r \vee K_{s-r}^{\ell}\big]
= \big(1 \pm  o_{\eps}(1) \big) 
\prod_{1 \leq i < j \leq s}
\big(d_F(V_i, V_j)\big)^{\lambda_i \lambda_j}
\geq
\big(1 - o_{\eps}(1)\big)
d_0^{\binom{r}{2} + r (s - r) \ell +  \binom{s-r}{2}\ell^2 }
$$
is big, proving  
Conclusion~(ii).   
To prove~\eqref{eqn:RCLii}, 
fix $\{u_i, u_j\} \neq \{v_{i'}, v_{j'} \} \in \binom{U}{2}_{\times}$ with 
$(u_i, u_j, v_{i'}, v_{j'}) \in U_i \times U_j \times U_{i'} \times U_{j'}$  
for $i < j$ and $i' < j'$.  
Then $(i, j), (i', j') \in E(D_F)$ 
from~\eqref{eqn:4.17.2024.6:11p}.  
\bigskip 

\noindent {\sc Case 1 ($\{i, j\} \cap \{i', j'\} = \emptyset$).}  
Since $E_F(u_i, V_j)$ is rainbow, 
$$
{\mathbb P} \big[\, c (\{u_i, u_j\}) = c(\{u_{i'}, u_{j'}\}) \, \big] 
\leq 
\tbinom{m-1}{\lambda_j - 1} \tbinom{m}{\lambda_j}^{-1} = \tfrac{\lambda_j}{m}
\leq \tfrac{\ell}{m}.  
$$
\smallskip 

\noindent {\sc Case 2 ($j = i'$, $j' = k$).}
The desired probability is at most 
$$
{\mathbb P} \Big[\, c(\{u_i, u_j\}) = c(\{v_j, v_k\}) \, \Big|  \,  
u_j \neq v_j \Big]
+ 
{\mathbb P} \Big[\, c(\{u_i, u_j\}) = c(\{v_j, v_k\}) \, \Big|  \,  
u_j = v_j \Big].     
$$
The first probability is,
by $E_F(u_i, V_j)$ being rainbow,   
at most 
$\tbinom{m-2}{\lambda_j - 2} \tbinom{m-1}{\lambda_j -1}^{-1} = 
\tfrac{\lambda_j - 1}{m-1} \leq \tfrac{2\lambda_j}{m}\leq \tfrac{2\ell}{m}$.  
The second probability is, by  
$c(\{v_j, v_k\})$ appearing at most $\tau m$ times on $E_F(v_j, V_i)$,   
at most 
$\tau m \tbinom{m-1}{\lambda_i - 1} \tbinom{m}{\lambda_i}^{-1} = \tau \lambda_i 
\leq \tau \ell$.    
\smallskip 

\noindent {\sc Case 3 ($i = i'$, $j = j'$).}
The desired probability is at most 
$$
{\mathbb P} \Big[\, c(\{u_i, u_j\}) = c(\{v_i, v_j\} ) \, \Big|  \,  
A \, \Big]  
+ 
{\mathbb P} \Big[\, c(\{u_i, u_j\}) = c(\{v_i, v_j\}) \, \Big|  \, B \, \Big] 
+ 
{\mathbb P} \Big[\, c(\{u_i, u_j\}) = c(\{v_i, v_j\}) \, \Big|  \,  
C \, \Big]
$$
where $A$, $B$, and $C$ are the respective events that $u_i \neq v_i$ and $u_j \neq v_j$, that $u_i = v_i$, and that $u_j = v_j$.  
The second probability is zero because 
$E_F(u_i, V_j)$ is rainbow and     
the third probability is zero
because 
$E_F (u_j, V_i)$
is rainbow from 
$u_i \neq v_i \in U_i$
giving $i \not\in \{1, 3, \dots, 2r - 1\}$ and $(j, i) \in E(D_F)$.
The first probability is, by 
$E_F(u_i, V_j)$ being rainbow,   
at most $\tbinom{m-2}{\lambda_j - 2} \tbinom{m-1}{\lambda_j -1}^{-1} = 
\tfrac{\lambda_j - 1}{m-1} \leq \tfrac{2\lambda_j}{m}\leq \tfrac{2\ell}{m}$.

\section{Proof of Theorem~\ref{thm:rainbowKr} }
Fix $\gamma > 0$ and $\ell, r, s \in \mathbb{Z}$ with $\ell \geq 1$, $r \geq 0$, and $s 
\geq \max \{1 + 2r, 2\}$.  Fix also $d_0 \in (0, 1)$ and let $\eps > 0$ be sufficiently 
small whenever needed.  
Let $(G, c)$ be a sufficiently large $n$-vertex edge-colored graph with 
vertex set 
$V = V(G)$.  
Construct
a $(G, c, \sqrt{n})$-digraph $D$ of $(G, c)$, and for $d_0, \eps > 0$ and $t_0 = 1$, 
let $V = V_0 \, \dot\cup \, V_1 \, \dot\cup \, \dots \, \dot\cup \, V_t$ and $D' \subseteq D$
be those objects guaranteed by 
Theorem~\ref{thm:degree_form}.  The standard multigraph $M' = M(D' - V_0)$ of $D' - V_0$ satisfies 
\begin{multline}  
\label{eqn:4.22.2024.12:36p}
e(M') = e(D' - V_0) 
\geq 
\sum_{v \in V \setminus V_0}
\big(d_D^+(v) - |V_0|\big) 
\geq 
- \eps n^2 + 
\sum_{v \in V \setminus V_0}
d_D^+(v)  \\
\stackrel{\text{\tiny Obs.\ref{obs:GtoDoutdeg}}}{\ge}
- \eps n^2 + 
\sum_{v \in V \setminus V_0}
\big(
d_G^c(v) - \big\lfloor \tfrac{d_G(v)}{1 + \sqrt{n}} \big\rfloor \big)
\geq 
- 2 \eps n^2 + 
\sum_{v \in V \setminus V_0}
d_G^c(v) 
\geq 
- 3 \eps n^2 + 
\sum_{v \in V}
d_G^c(v),  
\end{multline}  
which is $n \cdot \mathbb{E}\big[d_G^c(v)\big] - 3\eps n^2$.  
Respectively, 
$$
e(M') 
\stackrel{\eqref{eqn:4.22.2024.12:36p}}{\geq}
n \cdot \mathbb{E}\big[d_G^c(v)\big] - 3\eps n^2  
\quad \text{is} \quad 
\left\{
\begin{array}{lll}
\stackrel{\text{{\rm (i)}}}{\geq}  
\big(1 - \tfrac{1}{s-1} + \gamma - 3\eps \big) n^2 
& > &
\big(1 - \tfrac{1}{s-1}  \big) n^2 
\smallskip 
\\ 
\stackrel{\text{{\rm (ii)}}}{\geq}  
\big(1 - \tfrac{1}{2(s-1-r)} + \gamma - 3\eps \big) n^2 
& > &
\big(1 - \tfrac{1}{2(s-1 -r)}  \big) n^2  
\end{array}
\right.
$$
so 
Theorem~\ref{lem:multigraphturan}
guarantees that
$M'$ and hence $D'$ each contain ${\mathbb K}_s - {\mathcal M}_q$ for, respectively, $0 \leq q \leq s / 2$ and $0 \leq q \leq r$.    
Lemma~\ref{lem:MDr_to_rainbowKr}
(cf.~Remark~\ref{rem:embedding}) now guarantees that $(G, c)$
admits a rainbow $K_q \vee K^{\ell}_{s - q}$
for these same $q$.  Since
$K_q \vee K^{\ell}_{s - q}$
contains $K_{q+1} \vee K^{\ell}_{s - (q + 1)}$, 
Theorem~\ref{thm:rainbowKr} 
holds.


\section{Proof of Theorem~\ref{thm:rainbowKrExact}}
Theorem~\ref{thm:rainbowKrExact} follows quickly from the stability method.  
For that, and 
for $\beta > 0$, 
we say that 
an $n$-vertex graph $G = (V, E)$ is {\it $(K_s, \beta)$-extremal} when 
some 
$V_1, \dots, V_{s-1} \subset V$ satisfy
$|V_1|, \dots, |V_{s-1}| \geq \tfrac{n}{s-1} - \beta n$
and 
$e_G(V_1), \dots, e_G(V_{s-1}) < \beta n^2$.  
We will need the following lemmas.  




\begin{restatable}[Non-extremal Lemma]{lemma}{nonextremalKr}
\label{lem:nonextremal_rainbow_Kr}
Fix $s \in \mathbb{N}$, $\beta > 0$, and a sufficiently small $\alpha > 0$.  A sufficiently large $n$-vertex
edge-colored graph $(F, c_F)$ 
is $(K_s, \beta)$-extremal whenever  
\begin{enumerate}
\item[{\rm (i)}] $(F, c_F)$ is edge-minimal, 
\item[{\rm (ii)}]  $\delta^c(F) \ge \big(1 - \tfrac{1}{s-1} - \alpha\big)n$,  
\item[{\rm (iii)}] and $(F, c_F)$ admits no rainbow $K_s$.    
\end{enumerate}
\end{restatable}

\begin{restatable}[Extremal Lemma]{lemma}{extremalKr}
\label{lem:extremal_rainbow_Kr}
Fix $s \in \mathbb{N}$ and a sufficiently small $\beta > 0$.
If 
a sufficiently large $n$-vertex 
edge-colored graph $(F, c_F)$ satisfies 
that 
\begin{enumerate}
    \item[{\rm (I)}]
    $F$ is $(K_s, \beta)$-extremal, 
    \item[{\rm (II)}]  
    $\delta^{c_F}(F) \geq \big(1 - \tfrac{1}{s-1}\big)n$, 
    \item[{\rm (III)}]
    and $(F, c_F)$ admits no rainbow $K_s$, 
\end{enumerate}
then $F = K^L_{s-1}$ for $L = n/(s-1) \in \mathbb{N}$ and $c_F$ properly colors $E(F)$.  
\end{restatable}

\noindent  We defer the proofs of 
Lemmas~\ref{lem:nonextremal_rainbow_Kr} 
and~\ref{lem:extremal_rainbow_Kr} 
to subsequent sections.  


\subsection*{Proof of Theorem~\ref{thm:rainbowKrExact}}  
Fix an integer $s \geq 2$.
We fix $\beta > 0$ small for 
Lemma~\ref{lem:extremal_rainbow_Kr}, and then fix $\alpha > 0$
small for 
Lemma~\ref{lem:nonextremal_rainbow_Kr}.  
Let $(G, c)$ be a large $n$-vertex edge-colored graph 
satisfying the hypotheses of Theorem~\ref{thm:rainbowKrExact}, i.e., that   
$\delta^c(G) \geq \big(1 - \tfrac{1}{s-1}\big)n$
and $(G, c)$ has no rainbow $K_s$.  
We apply Lemma~\ref{lem:nonextremal_rainbow_Kr}  
to a fixed and guaranteed spanning subgraph $F \subseteq G$  
where $\big(F, c_F\big)$
is edge-minimal and color-degree preserving.   
Here, $(F, c_F)$ meets 
the hypotheses of 
Lemma~\ref{lem:nonextremal_rainbow_Kr}  
because~(i) 
holds by construction, 
(ii) 
holds by 
\begin{equation}
\label{eqn:5.2.2024.11:15a}
\delta^{c_F}(F) = \delta^c(G) \geq 
\big(1 - \tfrac{1}{s-1}\big)n
> 
\big(1 - \tfrac{1}{s-1} - \alpha \big)n,   
\end{equation}
and~(iii) is inherited from $(G, c)$.  
Lemma~\ref{lem:nonextremal_rainbow_Kr}  
guarantees that $(F, c_F)$ is $(K_s, \beta)$-extremal, so 
we next apply 
Lemma~\ref{lem:extremal_rainbow_Kr}
to it.  Here, $(F, c_F)$ 
meets 
the hypotheses of 
Lemma~\ref{lem:extremal_rainbow_Kr}
because~(I) holds from 
Lemma~\ref{lem:nonextremal_rainbow_Kr}, (II)   
holds 
from~\eqref{eqn:5.2.2024.11:15a}, and~(III) is inherited from $(G, c)$.  
Lemma~\ref{lem:extremal_rainbow_Kr}
guarantees that $F = K^L_{s-1}$ for $L = \tfrac{n}{s-1} \in \mathbb{N}$ and that $c_F$ properly colors $E(F)$.  
To conclude the proof of 
Theorem~\ref{thm:rainbowKrExact}, we 
use Observation~\ref{obs:proper_to_rainbow2} 
to show that the spanning subgraph $F$ is, in fact, $G$.   
For that, 
write $F = K[V_1, \dots, V_{s-1}]$
for a partition 
$V(F) = V_1 \, \dot\cup \, \dots \, \dot\cup \, V_{s-1}$ 
satisfying that $|V_1| = \dots = |V_{s-1}| = L$.
Fix, if possible, $uv \in E(G) \setminus E(F)$
and let, w.l.o.g., $u, v \in V_{s-1}$  
for $A = V_{s-1}$ and $B = 
V_1 \, \dot\cup \, \dots \, \dot\cup \, V_{s-2}$.    
Here, $(F, c_F)$ meets 
the hypotheses of 
Observation~\ref{obs:proper_to_rainbow2} 
because 
$F[A]$ contains the rainbow $uv$ (see~(i) there), 
$F[B] = K^L_{s-2} = K[V_1, \dots, V_{s-2}]$ is properly colored by $c_F$ (see~(ii) there), 
and $F[A, B] = K[A, B]$ is properly colored by $c_F$ (see~(iii) there).   
Observation~\ref{obs:proper_to_rainbow2} 
guarantees that $(F, c_F)$ admits a rainbow $K_2 \vee K_{s-2} = K_s$, a contradiction.



\section{Proof of 
Lemma~\ref{lem:nonextremal_rainbow_Kr}}   
\label{sec:nonext}
We prove Lemma~\ref{lem:nonextremal_rainbow_Kr} by reducing 
it 
to a proposition 
on directed graphs (itself proven in a moment).   


\begin{proposition}
\label{prop:triangleKr}
Fix an integer $s \geq 2$, a $\beta_0 > 0$, 
and a sufficiently small 
$\alpha_0 = \alpha_0(s, \beta_0) > 0$.  
Let $D_0$ be a sufficiently large $m$-vertex directed graph satisfying  
$\delta^+(D_0) \geq \big(1 - \tfrac{1}{s-1} - \alpha_0\big)m$ and 
that, for each $0 \leq r \leq s/2$ an integer, 
$D_0$ contains no copies of 
either 
$\mathbb{K}_s - \triangle_{\%} - \mathcal{M}_r$  
or $\mathbb{K}_s - \mathcal{M}_r$.      
Then, the underlying simple graph $G_0 = G(D_0)$ is $(K_s, \beta_0)$-extremal.  
\end{proposition}

It is easy to check
that 
Proposition~\ref{prop:triangleKr}
implies 
Lemma~\ref{lem:nonextremal_rainbow_Kr}.

\subsection*{Proof of 
Lemma~\ref{lem:nonextremal_rainbow_Kr}}
Fix $s \in \mathbb{N}$ and $\beta > 0$.
We define auxiliary constant $\beta_0 = \beta / 6$ and take $\alpha_0 = \alpha_0(s, \beta_0) > 0$ 
and $\alpha = \alpha_0 / 7$ 
sufficiently small.  
We next fix even smaller $d_0 = d(s, \beta_0, \alpha) > 0$ and $\eps = \eps(s, \beta_0, \alpha, d_0) > 0$  
for applications of 
Theorem~\ref{thm:degree_form} and  
Lemma~\ref{lem:MDr_to_rainbowKr}.  
Now, 
let $(F, c_F)$ be a large $n$-vertex edge-colored graph 
satisfying hypothesies~(i)--(iii) of  
Lemma~\ref{lem:nonextremal_rainbow_Kr}, where for notational simplicity we now write $(G, c) = (F, c_F)$.  
Let $D$ be a $(G, c, \sqrt{n})$-digraph of $(G, c)$.  
For $t_0 = 1$ and 
$d_0, \eps > 0$ above, 
let $V(D) = V_0 \, \dot\cup \, V_1 \, \dot\cup \, \dots \, \dot\cup \, V_t$
and $D' \subseteq D$ be those objects guaranteed by 
Theorem~\ref{thm:degree_form}.  
Lemma~\ref{lem:MDr_to_rainbowKr}
and Remark~\ref{rem:embedding} guarantee that 
$D_0 = D' - V_0$ contains copies of 
neither 
$\mathbb{K}_s - \triangle_{\%} - {\mathcal M}_r$
nor 
$\mathbb{K}_s - {\mathcal M}_r$
lest $(G, c)$ admits a rainbow $K_s$.    
We therefore apply 
Proposition~\ref{prop:triangleKr}
to $D_0$ 
with $m = n - |V_0|$.  
To see that 
$\delta^+(D_0)$ is large enough, we first note that 
\begin{equation}
    \label{eqn:5.8.2024.4:44p}
\delta^+(D) 
\stackrel{\text{\tiny 
Obs.\ref{obs:GtoDoutdeg}}}{\geq} 
\delta^c(G) - \sqrt{n} 
\stackrel{\text{(ii)}}{\geq}  
\big(1 - \tfrac{1}{s-1} - \alpha \big)n - \sqrt{n}  
\geq 
\big(1 - \tfrac{1}{s-1} - 2\alpha \big)n
\end{equation}  
to which 
Theorem~\ref{thm:degree_form} adds 
\begin{multline*}  
\delta^+(D_0) = \delta^+(D' - V_0) 
\geq \delta^+(D) - (3 d_0 + \eps)n - |V_0| 
\geq \delta^+(D) - (3 d_0 + 2 \eps)n \\ 
    \stackrel{\eqref{eqn:5.8.2024.4:44p}}{\geq}  
\big(1 - \tfrac{1}{s-1} - 2\alpha - 3d_0 - 2\eps \big)m
\geq 
\big(1 - \tfrac{1}{s-1} - 7\alpha \big)m  
= 
\big(1 - \tfrac{1}{s-1} - \alpha_0 \big)m.      
\end{multline*}  
Proposition~\ref{prop:triangleKr}
guarantees that the simple graph $G_0$ underlying $D_0$ 
is 
$(K_s, \beta_0)$-extremal, so let   
$U_1, \dots, U_{s-1} \subseteq V(G_0) = V(G) - V_0 \subseteq V(G)$ be pairwise disjoint subsets satisfying 
$$
|U_1|, \dots, |U_{s-1}| 
\geq 
\tfrac{m}{s-1} - \beta_0 m 
= 
\tfrac{n - |V_0|}{s-1} - \beta_0 (n - |V_0|)     
\geq \tfrac{n}{s-1} - (\beta_0 + \eps)n 
\geq \tfrac{n}{s-1} - 2\beta_0n   
> \tfrac{n}{s-1} - \beta n   
$$
and 
$e_{G_0}(U_1), \dots, e_{G_0}(U_{s-1}) \leq \beta_0 m^2 \leq \beta_0 n^2$.  
The simple graph $G_D$ underlying $D$ therefore satisfies 
$$
e_{G_D}(U_1), \dots, e_{G_D}(U_{s-1}) \leq \beta_0 n^2 
+ 
|V_0| n 
+ 
n (3d_0 + \eps)n 
= 
(\beta_0 + 3d_0 + 2\eps) n^2 
\leq 6 \beta_0 n^2  
= 
\beta n^2, 
$$
so $G_D$ is $(K_s, \beta)$-extremal.  
Since $(G, c)$ is edge-minimal, 
every $\{u, v\} \in E(G)$ yields $(u, v)$ or $(v, u)$ in $E(D)$
so $G = G_D$ is the simple graph underlying $D$ (and is hence $(K_s, \beta)$-extremal).

\subsection*{On the proof of 
Proposition~\ref{prop:triangleKr}}  
We prove Proposition~\ref{prop:triangleKr} by reducing it to one     
on multigraphs.




\begin{restatable}{proposition}{nonext}
\label{prop:nonext}
Fix an integer $s \geq 2$, a $\beta > 0$, and a sufficiently small $\alpha = \alpha(s, \beta) > 0$. 
Fix a sufficiently large $n$-vertex standard multigraph $M$ 
satisfying that 
$\delta(M) \geq \big(2 - \tfrac{2}{s-1} - \alpha\big)n$  
and that, for each $0 \leq r \leq s / 2$ an integer,  
$M$ 
contains no copies 
of $\mathbb{K}_s - {\mathcal M}_r$.
There exist pairwise disjoint $V_1, \dots,  V_{\ell} \subseteq V(M)$, for  
$\ell \leq s - 1$, 
where each 
class $V_i$ satisfies at least one of the following
properties:  
\begin{enumerate}
\item[$(P_1)$]  
$V_i$ is independent and 
$|V_i| \geq \big(\tfrac{1}{s-1} - \beta\big)n$; 
\item[$(P_2)$]  
$V_i$ has no three points spanning five or more edges of $M$ and 
$|V_i| \geq \big(\tfrac{2}{s-1} - \beta\big) n$.
\end{enumerate}
\end{restatable}

\noindent  The proof of Proposition~\ref{prop:nonext} is non-trivial.  We devote special attention 
to it in the next section.  For our current purposes,  
we will also need the following fact (itself proven in a moment).
\begin{fact}
\label{fact:cyclic_triangle}
Fix $0 < 25 \alpha < 1$.  
Let $D$ be a 
sufficiently large $n$-vertex digraph 
satisfying $\delta^+(D) \geq \big(\tfrac{1}{2} - \alpha\big)n$ 
and let $D$ 
have no cyclically directed triangles.   
Then 
the underlying graph $G = G(D)$ of $D$ is $\big(K_3, 4\alpha^{1/2}\big)$-extremal.  
\end{fact}

We next infer Proposition~\ref{prop:triangleKr}
from the results above.  


\subsection*{Proof of Proposition~\ref{prop:triangleKr}}  
Fix an integer $s \geq 2$ and fix $\beta = \beta_0 > 0$.  We take $\alpha = \alpha_0 = \alpha_0(\beta) > 0$ small enough to enable the calculations
below which feature, for notational simplicity,  
multiple 
undisclosed but vanishing quantities $o_{\alpha}(1) \to 0^+$ in $\alpha \to 0^+$.  
Fix a directed graph $D = D_0$ 
on a large number $n = m$ of vertices 
satisfying the hypotheses of 
Proposition~\ref{prop:triangleKr}.    
By these hypotheses, 
the underlying standard multigraph $M = M(D)$ of $D$ satisfies
$$
e(M) = e(D) = \sum_{v \in V} d_D^+(v) 
\geq 
\big(1 - \tfrac{1}{s-1} - \alpha\big) n^2.  
$$
These same hypotheses say that 
$M$ 
contains no $\mathbb{K}_s - \mathcal{M}_r$, so 
every induced submultigraph $M_0 \subseteq M$ satisfies  
$$
e(M_0) 
\stackrel{\text{\tiny Thm.\ref{lem:multigraphturan}}}{\leq}  
\big(1 - \tfrac{1}{s-1}\big) |V(M_0)|^2.  
$$
Observation~\ref{obs:mindeg} guarantees an $m$-vertex such $M_0 = M_0'$ satisfying 
\begin{equation}
\label{eqn:5.21.2024.5:45}
m \geq 
\big(1 - o_{\alpha}(1) \big)n 
\qquad \text{and} \qquad 
\delta(M_0) \geq 
2
\big(1- \tfrac{1}{s-1} - o_{\alpha}(1) \big) m
\stackrel{\eqref{eqn:5.21.2024.5:45}}{\geq}
2
\big(1- \tfrac{1}{s-1} - o_{\alpha}(1) \big) n.    
\end{equation}  
Proposition~\ref{prop:nonext} 
now applies 
to $M_0$ to guarantee pairwise disjoint 
classes $V_1, \dots, V_{\ell} \subseteq V(M_0) \subseteq V(M)$, for 
$\ell \leq s - 1$, where each class $V_i$ 
satisfies at least one of $(P_1)$ or $(P_2)$ there.  
If all classes satisfy $(P_1)$, i.e., 
\begin{multline*}  
|V_1|, \dots, |V_{\ell}| \geq 
\big(\tfrac{1}{s-1} - o_{\alpha}(1) \big) m
\stackrel{\eqref{eqn:5.21.2024.5:45}}{\geq}
\big(\tfrac{1}{s-1} - \beta \big) n  
\quad \text{(and thus $\ell = s$)} \\
\text{and} \qquad 
e_{M_0}(V_1) = 
e_{M}(V_1) = \dots =  
e_{M_0}(V_{\ell}) = 
e_{M}(V_{\ell}) = 0 < \beta n^2,  
\end{multline*}
then the underlying simple graph $G = G(M)$ is immediately $(K_s, \beta)$-extremal.  
In a very similar way, if every class $V_j$ satisfying~$(P_2)$ also satisfies that $G[V_j]$ is $\big(K_3, o_{\alpha}(1)\big)$-extremal, viz.,  
\begin{equation}
\label{eqn:5.22.2024.12:46p}
|V_j| \geq 
\big(\tfrac{2}{s-1} - o_{\alpha}(1) \big)m  
\stackrel{\eqref{eqn:5.21.2024.5:45}}{\geq}
\big(\tfrac{2}{s-1} - o_{\alpha}(1) \big)n  
\end{equation}
and that, 
for some disjoint $T_j, U_j \subset V_j$, 
$$
|T_j|, |U_j| \geq \big(\tfrac{1}{2} - o_{\alpha}(1)\big) |V_j|
\stackrel{\eqref{eqn:5.22.2024.12:46p}}{\geq}  
\big(\tfrac{1}{s-1} - \beta \big)n
\qquad \text{and} \qquad 
e_G(T_j), \, e_G(U_j) \leq \beta n^2,   
$$
then $G = G(M)$ is immediately $(K_s, \beta)$-extremal.  
Thus, for sake of argument, we assume that some class $V_k$ satisfies $(P_2)$ and that $G[V_k]$ is not $\big(K_3, o_{\alpha}(1)\big)$-extremal.
We claim this same $V_k$ has the following additional properties 
(which we easily prove in a moment).  

\begin{claim}
\label{clm:5.21.2024.6:53p}
$|V_k| = \big(\tfrac{2}{s-1} \pm o_{\alpha}(1) \big)m$ and 
$\delta^+(D[V_k]) \geq \big(\tfrac{1}{2} - o_{\alpha}(1)\big) |V_k|$.  
\end{claim}  

Continuing, 
since $G[V_k]$ is not $\big(K_3, o_{\alpha}(1)\big)$-extremal
but 
$\delta^+(D[V_k]) \geq \big(\tfrac{1}{2} - o_{\alpha}(1)\big) |V_k|$, 
Fact~\ref{fact:cyclic_triangle}
guarantees a cyclic triangle $(x, y), (y, z), (z, x) \in E(D[V_k])$.  
Consider now the simple graph $H_0 = H(M_0)$ of heavy edges from $M_0$.  
Since $V_k$ satisfies~$(P_2)$ and~\eqref{eqn:5.21.2024.5:45}, 
the set $\overline{V}_k = V(M_0) \setminus V_k$ satisfies, e.g.,  
$$
d_{H_0}\big(x, \overline{V}_k\big)
\stackrel{\text{{\tiny Obs.\ref{obs:extmultiHdeg}}}}{\geq}
\big|\overline{V}_k\big| - m \cdot o_{\alpha}(1).  
$$
Consequently, the set $W = N_{H_0}(x) \cap N_{H_0}(y) \cap N_{H_0}(z) \cap \overline{V}_k$ satisfies 
$$
|W| \geq \big|\overline{V}_k\big| - 3m \cdot o_{\alpha}(1) 
\stackrel{\text{\tiny{Clm.\ref{clm:5.21.2024.6:53p}}}}{\geq}  
\big(1 - \tfrac{2}{s-1} - o_{\alpha}(1)\big) m  
= 
\big(\tfrac{s-3}{s-1} - o_{\alpha}(1)\big) m  
$$
which, 
again using~\eqref{eqn:5.21.2024.5:45}, 
guarantees 
$$
\delta\big(M[W]\big) 
\stackrel{\text{\tiny{Obs.\ref{obs:deg}}}}{\geq}  
2 
\big(\tfrac{s-4}{s-3} - o_{\alpha}(1) \big) |W|
> 
2 
\big(\tfrac{s-5}{s-4}\big)
|W|
= 
2 \big(1 - \tfrac{1}{s-4}\big) |W|.  
$$
Consequently, 
$$
e_{M_0}(W) = (1/2) \sum_{w \in W} d_{M_0}(w, W) 
>  
\big(1 - \tfrac{1}{s-4}\big) |W|^2
$$
so Theorem~\ref{lem:multigraphturan}
guarantees 
that $M_0[W] = M[W]$ contains a $\mathbb{K}_{s-3} - \mathcal{M}_r$.  
Together with the cyclic triangle $(x, y, z)$, the directed graph $D$ 
contains $\mathbb{K}_s - \triangle_{\%} - \mathcal{M}_r$, contradicting our hypothesis.  
\hfill $\Box$  

\subsubsection*{Proof of Claim~\ref{clm:5.21.2024.6:53p}}  
First, 
the lower bound 
$|V_k| \geq \big(\tfrac{2}{s-1} - o_{\alpha}(1)\big)m$ is from~$(P_2)$.  
For the corresponding upper bound, that 
$M_0[V_k]$ has no $\mathbb{K}_s - \mathcal{M}_r$ guarantees 
$$
\sum_{v_k \in V_k} d_{M_0[V_k]} (v_k) = 2 e_{M_0}(V_k) 
\stackrel{\text{\tiny Thm.\ref{lem:multigraphturan}}}{\leq}  |V_k|^2
$$
so some $u_k \in V_k$ satisfies $d_{M_0[V_k]}(u_k) \leq |V_k|$.  
Now, 
$d_{M_0}(u_k)  
= 
d_{M_0[V_k]}(u_k)  
+ 
d_{M_0 \setminus M_0[V_k]}(u_k)$   
satisfies 
$$
2 \big( 1 - \tfrac{1}{s-1} - o_{\alpha}(1) \big) 
m
\stackrel{\eqref{eqn:5.21.2024.5:45}}{\leq}
\delta(M_0) \leq 
d_{M_0}(u_k)  
\leq |V_k| +  2 (m - |V_k|) = 2m - |V_k|,   
$$
so 
$|V_k| \leq \big(\tfrac{2}{s-1} + o_{\alpha}(1)\big)m$.  
Second, 
\begin{multline*}
\delta^+\big(D[V_k]\big) \geq 
\delta^+(D) 
- 
\big(n - |V_k|\big)
\stackrel{\text{\tiny hyp}}{\geq}
\big(1 - \tfrac{1}{s-1} - \alpha\big) n - 
\big(n - |V_k|\big)
= 
|V_k| - 
\big(\tfrac{1}{s-1} + \alpha \big)n  \\ 
\stackrel{\eqref{eqn:5.22.2024.12:46p}}{\geq} 
\big(\tfrac{1}{s-1} - o_{\alpha}(1) \big)n    
= 
\big(\tfrac{1}{s-1} - o_{\alpha}(1) \big)\tfrac{n}{|V_k|} |V_k|     
\stackrel{\text{\tiny Clm.\ref{clm:5.21.2024.6:53p}}}{\geq}    
\big(\tfrac{1}{s-1} - o_{\alpha}(1)\big)
\big(\tfrac{2}{s-1} + o_{\alpha}(1)\big)^{-1} |V_k|, 
\end{multline*}  
which is at least 
$\big(\tfrac{1}{2} - o_{\alpha}(1)\big) |V_k|$.

\subsection*{Proof of Fact~\ref{fact:cyclic_triangle}}  
Fix $0 < 25 \alpha < 1$.  
Let $D$ be a sufficiently large $n$-vertex directed graph satisfying
$\delta^+(D) \geq \big(\tfrac{1}{2} - \alpha\big)n$ and let $D$ have no 
cyclically directed triangles.  
From these hypotheses, 
both $D$ and its underlying standard multigraph $M(D)$ satisfy 
$$
e(M) = e(D) = \sum_{v \in V(D)} d_D^+(v) \geq \big(\tfrac{1}{2} - \alpha\big)n^2  
$$
while neither 
contains $\mathbb{K}_3 - \mathcal{M}_1$, 
whence 
every induced $M_0 \subseteq M$ satisfies 
$$
e(M_0) 
\stackrel{\text{\tiny Thm.\ref{lem:multigraphturan}}}{\leq}  
(1/2) |V(M_0)|^2.   
$$
Observation~\ref{obs:mindeg} then guarantees an $m$-vertex multigraph $M_0 \subseteq M$ satisfying 
\begin{equation}
    \label{eqn:5.23.2024.12:54p}
m \geq \big(1 - \alpha^{1/2}\big)n \qquad \text{and} \qquad 
\delta(M_0) \geq 
\big(1 - 2\alpha^{1/2}\big) m
    \stackrel{\eqref{eqn:5.23.2024.12:54p}}{\geq}  
\big(1 - 3\alpha^{1/2}\big)n.  
\end{equation}  
Now, 
to prove that the underlying graph $G = G(D)$ of $D$ is 
$\big(K_3, 4 \alpha^{1/2}\big)$-extremal, 
we consider the 
graph $H = H(D)$ of heavy edges
of $D$ and the following two cases.   
\bigskip 

\noindent  {\sc Case 1 $\big(\exists \, v_1 \in V(D): \, d_H(v_1) \geq \big(\tfrac{1}{2} - \alpha^{1/2}\big)n\big)$}.
Fix $v_0 \in V(M_0) \cap N_H(v_1)$.  
The sets $N_H(v_0)$ and $N_H(v_1)$ will witness  
the desired 
extremality of $G$.  
Indeed, 
these sets are 
each 
independent 
and disjoint.  Moreover, 
$N_H(v_1)$ is large enough in this case.  To see that $N_H(v_0)$ is also large enough, we use that 
$N_G(v_1)$ and $N_H(v_0)$ are disjoint    
and 
the simple identity $d_H(v) + d_G(v) = d_M(v)$ holding in standard multigraphs:    
\begin{multline*}
|N_H(v_0)| = 
d_H(v_0) = 
d_M(v_0) - d_G(v_0) 
\geq 
d_{M_0}(v_0) - d_G(v_0)   \\
\stackrel{\text{\tiny{disj}}}{\geq} 
d_{M_0}(v_0) - \big(n - d_H(v_1)\big)   
    \stackrel{\eqref{eqn:5.23.2024.12:54p}}{\geq}
\big(1 - 3\alpha^{1/2}\big)n  
+ 
d_H(v_1) - n 
\stackrel{\text{\tiny{Case~1}}}{\geq} 
\big(\tfrac{1}{2} - 4 \alpha^{1/2}\big)n.  
\end{multline*}
\smallskip 

\noindent  {\sc Case 2 $\big(\forall \, v \in V(D), \, \, 
d_H(v) <\big(\tfrac{1}{2} - \alpha^{1/2}\big)n\big)$}.
Fix $u_1 \in V(D)$ satisfying 
$d_D^-(u_1) \geq \delta^+(D)$.  We claim that 
\begin{equation}
\label{eqn:5.27.2024.11:56a}
\big|N^+_D(u_1) \setminus N^-_D(u_1)\big| \leq 2n - 4 \delta^+(D).  
\end{equation}  
If true, 
the identity $d^+_D(u_1) = \big|N^+_D(u_1) \setminus N^-_D(u_1)\big| + d_H(u_1)$ 
yields 
\begin{multline*}
d_H(u_1) = 
d_D^+(u_1) - 
\big|N^+_D(u_1) \setminus N^-_D(u_1)\big| 
\stackrel{\eqref{eqn:5.27.2024.11:56a}}{\geq}  
d_D^+(u_1)  
+4\delta^+(D) - 2n \\
\geq
5\delta^+(D) - 2n 
\stackrel{\text{\tiny hyp}}{\geq}
5\big(\tfrac{1}{2} - \alpha\big)n - 2n = 
\big(\tfrac{1}{2} - \alpha^{1/2}\big)n
\end{multline*}
which 
contradicts the condition of Case~2.  
To prove~\eqref{eqn:5.27.2024.11:56a}, 
define $I_1 = N^-_D(u_1)$ and $O_1^* = N_D^+(u_1) \setminus I_1$.  
Then 
$$
\sum_{u \in O_1^*} d^+_D\big(u, O_1^*\big) = 
e_D\big(O_1^*\big) = 
e_M\big(O_1^*\big)  
    \stackrel{\text{\tiny Thm.\ref{lem:multigraphturan}}}{\leq}
    (1/2) 
    \big|O_1^*\big|^2 
$$
because 
$M\big[O_1^*\big]$
has no copies of $\mathbb{K}_3 - \mathcal{M}_1$.  
We therefore fix, for sake of argument, $u_0 \in O_1^*$  
satisfying 
\begin{equation}
   \label{eqn:5.27.2024.1:09p} 
d_D^+\big(u_0, O_1^*\big) \leq (1/2) \big|O_1^*\big|.  
\end{equation}
Since $O_1^*, I_1 \subseteq V = V(D)$ are disjoint
and $D$ has no cyclic triangles, 
\begin{multline*}
d_D^+(u_0) = 
d_D^+\big(u_0, O_1^*\big) + 
d_D^+(u_0, I_1) + 
d_D^+\Big(u_0, V \setminus \big(O_1^* \, \dot\cup \, I_1 \big)\Big)
= 
d_D^+\big(u_0, O_1^*\big) + 
d_D^+\Big(u_0, V \setminus \big(O_1^* \, \dot\cup \, I_1 \big)\Big) \\
\leq 
d_D^+\big(u_0, O_1^*\big) + 
n - 1 - 
\big|O_1^*\big| - |I_1|
   \stackrel{\eqref{eqn:5.27.2024.1:09p}}{\leq}  
n - (1/2)
\big|O_1^*\big| - |I_1|
= 
n - (1/2)
\big|O_1^*\big| - 
d_D^-(u_1).  
\end{multline*}
Now, \eqref{eqn:5.27.2024.11:56a}
follows from 
$d_D^-(u_1), \, d_D^+(u_0) \geq \delta^+(D)$.

\section{Proof of Proposition~\ref{prop:nonext}}
We begin this section by establishing some simplifying 
notation, terminology, and considerations.    
First, 
from the hypothesis Proposition~\ref{prop:nonext}, 
a standard multigraph $M$ contains 
$\mathbb{K}_s -\mathcal{M}_r$ 
for some 
integer 
$0 \leq r \leq s/2$
if, and only if, it contains 
$\mathbb{K}_s - \mathcal{M}_{\floor{s/2}}$.
(These are not-necessarily induced containments.)
As a single-parameter structure, 
we henceforth simplify 
$\mathbb{K}_s - \mathcal{M}_{\floor{s/2}}$
to $\MK_s$.  
Thus, 
the hypothesis of 
Proposition~\ref{prop:nonext} includes that $M$ is an $n$-vertex 
$\MK_s$-free 
standard multigraph
satisfying $\delta(M) \geq \big(2 - \tfrac{2}{s-1} - \alpha\big) n$.  
For 
the corresponding conclusion, we say that $M$ is 
{\it $\big(\MK_s, \beta\big)$-extremal}
when 
some 
pairwise disjoint $V_1, \dots, V_{\ell} \subseteq V(M)$ 
with $\ell \leq s - 1$ have each class $V_i$ satisfying at least one 
of Properties $(P_1)$
or $(P_2)$
(of Proposition~\ref{prop:nonext}, 
where the standard multigraph
in $(P_2)$
with three points spanning five or more edges is precisely $\MK_3$).  
In sum, and with appropriately given constants, 
Proposition~\ref{prop:nonext} asserts that 
a large $n$-vertex $\MK_s$-free standard multigraph $M$ 
satisfying $\delta(M) \geq \big(2 - \tfrac{2}{s-1} - \alpha\big) n$
must be 
$\big(\MK_s, \beta\big)$-extremal.



\subsection*{The proof}  
We induct on $s \geq 3$, where for sake of argument we assume $s \geq 4$.  
In particular, 
with $\beta > 0$ given, 
we take $0 < \beta '' \ll \beta '  \ll \beta$
and let $\alpha'' > 0$ be inductively guaranteed for $\beta''$ and $s - 1$ or $s - 2$.  
We then take the promised constant $\alpha > 0$ 
to satisfy $0 \ll \alpha \ll \alpha''$.  
Inductively, 
 assume that $M$ is a large $n$-vertex $\MK_s$-free standard multigraph satisfying 
$\delta(M) \geq \big(2 - \tfrac{2}{s-1} - \alpha\big) n$.  
Assume, on the contrary, 
that $M$ is not $(\MK_s, \beta)$-extremal.
We proceed with the following claim.



  \begin{claim}\label{clm1}
    If $U \subseteq V(M)$ such that $|U| \ge n/(s-1) - \beta' n$, then $e(G[U]) \neq 0$.
    Likewise, if $U \subseteq V(M)$ with $|U| \ge 2n/(s-1) - \beta' n$, then 
    $M[U]$ contains a $\MK_3$.
  \end{claim}
  \begin{proof}
    First, assume there exists $U \subseteq V(M)$ such that
    $|U| = n/(s-1) - \beta' n$ and $e(G[U]) = 0$.
    Let $v \in U$ and let $W := N_H(v)$.
    By Observation~\ref{obs:extmultiHdeg}, $|W| \ge (1 - 1/(s-1) - \beta) n$
    and since $M$ is $\MK_s$-free, $M[W]$ is $\MK_{s-1}$-free.
    These facts with Observation~\ref{obs:deg} and the induction hypothesis,
    imply that $M[W]$ is $(\MK_{s-1}, \beta'')$-extremal.
    So, $M$ is $(\MK_{s}, \beta)$-extremal, a contradiction.
Second, assume there exists $U \subseteq V(M)$ such that
    $|U| = 2n/(s-1) - \beta' n$ and $M[U]$ is $\MK_3$-free.
    By Observation~\ref{obs:deg}, there exists $uv \in E(G[U])$.
    Let $W := N_H(u) \cap N_H(v)$.
    By Observation~\ref{obs:extmultiHdeg}, $|W| \ge (1 - 2/(s-1) - \beta)n$
    and because $M$ is $\MK_s$-free, $M[W]$ is $K_{s-2}$-free.
    These facts with Observation~\ref{obs:deg} and the induction hypothesis,
    imply that $M[W]$ is $(\MK_{s-2}, \beta'')$-extremal.
    So, $M$ is $(\MK_{s}, \beta)$-extremal, a contradiction.
  \end{proof}

  We next divide the proof into cases.

\subsubsection*{{\sc Case 1} $\boldsymbol{\big(}$every $\MK_{s-1}$ in $M$ has at least one light edge$\boldsymbol{\big)}$}
  Lemma~\ref{lem:multigraphturan} guarantees from 
  $\delta(M) \ge (2 - 2/(s-1) - \alpha)n > (2 - 2/(s-2))n$ a $\MK_{s-2}$ in $M$.
  Fix such a copy $K$ maximizing $e(K)$.
  Suppose there exists $v \in V(G)$ with $d_M(v, K) = 2(s - 2)$.
  If $e(K) = 2\binom{s-2}{2}$, then $K + v$ is copy of $\MK_{s-1}$ with only heavy edges, 
  contradicting the case.
  If $e(K) < 2\binom{s-2}{2}$, then there exists $u \in V(K)$ with $d_M(u, K - u) < 2(s - 3)$
  while $d_M(v, K - u) = 2(s-3)$.
  Then $K' := K - u + v$ is a copy of $\MK_{s-2}$ with $e(K') > e(K)$, contradicting our choice of $K$.
  Thus, $d_M(v, K) \le 2s - 5$ for every $v \in V(G)$.
  Let $U := \{ u \in V(G) : d_M(u, K) = 2s - 5\}$.
  Then 
  \begin{equation*}
    2(s - 2) \big(\tfrac{s-2}{s-1} - \tfrac{\alpha}{2} \big)n \le (s-2) \delta(M) \le e_M(V(K), V(G)) \le |U| + (2s - 6)n,
  \end{equation*}
  so
  \begin{equation*}
    |U| \ge 2(s - 2)\big(\tfrac{s-2}{s-1} - \tfrac{\alpha}{2} \big)n - (2s-6)n = 
    2\Big(\tfrac{(s-2)^2 - (s-1)(s-3)}{s-1}\Big)n - (s-2) \alpha n \ge \tfrac{2n}{s-1} - \beta' n.
  \end{equation*}
  Claim~\ref{clm1} guarantees
  $\MK_3 = v_1v_2v_3$ in $M[U]$.
  W.l.o.g., let $v_1v_2$ and $v_2v_3$ be heavy.
  Recall for each $i \in [3]$ that $v_i \in U$ so $d_M(v_i, K) = 2s - 5$. 
  Therefore, there exists $v'_i \in V(K)$ so that the edge $v_iv'_i$
  is light and for every $v \in V(K) \setminus \{v'_i\}$ the edge $v_iv$ is heavy.
  Now, 
  every edge incident to $v'_i$ in $E(K)$ is heavy lest
  $K - v'_i + v_i$ is a $\MK_{r-2}$ contradicting our choice of $K$.
If at least two of $v_1',v_2', v_3'$
are distinct, 
  then there exists $\{i, j\} \in \{\{1, 2\}, \{2, 3\}\}$, 
 with 
  $v'_i\neq v'_j$.
 Then, since $v_iv_j$ is a heavy edge and every edge in $E(K)$ that is incident
  to $v'_i$ or $v'_j$ is heavy, $K + v_i + v_j$ is a copy of $\MK_s$, a contradiction.
  Otherwise, $v'_1=v'_2=v'_3 $ and $K - \{v'_1\} + \{v_1,v_2,v_3\}$ is
  a copy of $\MK_s$, which is also a contradiction.

\subsubsection*{{\sc Case 2} $\boldsymbol{\big(}$some $\MK_{s-1}$ in $M$ has all heavy edges$\boldsymbol{\big)}$}  
  Let $K$ be such a copy and let $v_1, \dotsc, v_{s-1}$ be an
  ordering of its vertices.
  Since $M$ is $\MK_s$-free, every $v \in V(M)$
  satisfies 
  $d_M(v, K) \le 2s - 4$.
  Therefore, set 
  \begin{equation*}
    V' := \{v \in V : d_M(v, K) = 2s - 4 \} 
  \end{equation*}
  so that 
  \begin{equation*}
    (s-1) \big(\tfrac{2(s-2)}{s-1} - \alpha\big)n \le  (s-1)\delta(G) \le e_M(V(K), V(G)) \le |V'| + (2s - 5)n
    \end{equation*}
  and 
  \begin{equation}\label{eq:sizeofVprime}
    |V'| \ge n - (s-1) \alpha n.
  \end{equation}
  Let 
  \begin{align*}
    &X_i := \{ v \in V' : d_M(v, v_i) = 0 \} \text{ for all $i \in [r-1]$ and} \\
    &Y_{i,j} := \{ v \in V' : d_M(v, v_i) = d_M(v,v_j) = 1 \} \text{ for all distinct $i,j \in [s-1]$.}
  \end{align*}
  Then $Y_{i,j} = Y_{j,i}$ and the sets 
  $X_i$ for $i \in [s-1]$ and $Y_{i,j}$ for $1 \le i < j \le s-1$ partition $V'$.
  Therefore,  
  \begin{equation}\label{eq:XiYikSum}
    \sum_{i \in [s-1]} \Big(2 |X_i| + \sum_{i \neq k} |Y_{i,k}| \Big) = 2|V'|.
  \end{equation}
 The $\MK_s$-freeness of $M$ will guarantee  
  the following simple facts:
  \begin{claim}\label{clm:simpfacts}
    The following hold for all distinct $i,j,k \in [s-1]$.
    \begin{enumerate}[label=(\roman*)]
      \item\label{nonexti} $E(G[X_i]) = \emptyset$.
      \item\label{nonextii} The bigraphs $H[X_i, Y_{i,j}]$ and $H[Y_{i,j}, Y_{i,k}]$ have no edges.
      \item\label{nonextiii} The graph $M[Y_{i,j}]$ is $\MK_3$-free.
    \end{enumerate}
  \end{claim}
  \begin{proof}
    If $uv \in E(G[X_i])$, then $K - v_i + u + v$ is a $\MK_{s}$.
    Similarly, if $uv \in E(H[X_i, Y_{i,j}]) \cup E(H[Y_{i,j}, Y_{i,k}])$, 
    then $K - v_i + u + v$ is a $\MK_{s}$.  This proves \ref{nonexti} and \ref{nonextii}.
    To see that \ref{nonextiii} holds, note that if $T$ is a $\MK_3$ in $M[Y_{i,j}]$, 
    then $K - v_i - v_j + T$ is a $\MK_s$. 
  \end{proof}

  \begin{claim}\label{clm:3}
    All $1 \leq i \neq j \leq r-1$ satisfy 
    $|X_i| + |X_j| + |Y_{i,j}| \le 2n/(r-1) - \beta' n$.
  \end{claim}
  \begin{proof}
    Let $U := X_i \cup X_j \cup Y_{i,j}$.
    Claim~\ref{clm:simpfacts}
    guarantees that 
    $M[U]$ is $\MK_3$-free
    and 
    Claim~\ref{clm1} implies that $|U| \le 2n/(s-1) - \beta' n$.
  \end{proof}

  \begin{claim}\label{clm:2}
    All $1 \leq i \leq s-1$ satisfy $2|X_i| + \sum_{k \neq i}|Y_{i,k}| \le 2n/(s-1) + \alpha n$.
  \end{claim}
  \begin{proof}
    First assume $X_i = \emptyset$.
    Claim~\ref{clm:simpfacts}
    guarantees that 
    $M[\bigcup_{k \neq i}Y_{i,k}]$ is $\MK_3$-free, so Claim~\ref{clm1}
    guarantees 
    \begin{equation*}
      2|X_i| + \sum_{k \neq i}|Y_{i,k}| = 
      \Big|\bigcup_{k \neq i}Y_{i,k}\Big|
      \le 2n/(s-1) - \beta' n \le 2n/(s-1) + \alpha n,
    \end{equation*}
    Otherwise, let $v \in X_i$. 
    Claim~\ref{clm:simpfacts}
    guarantees 
    $d_G(v, X_i) = d_H(v, \bigcup_{k \neq i}Y_{i,k}) = 0$ so  
    \begin{equation*}
      2|X_i| + \sum_{k \neq i}|Y_{i,k}| \le 2n - \delta(M) \le 2n/(s-1) + \alpha n. \qedhere
    \end{equation*}
  \end{proof}

  Note  
  from~\eqref{eq:sizeofVprime} and~\eqref{eq:XiYikSum} that 
  \begin{equation*}
    \sum_{i = 1}^{s-1} \Big(2|X_i| + \sum_{k \neq i}|Y_{i,k}|\Big) = 2|V'| \ge 2n - 2(s-1) \alpha n.
  \end{equation*}
  Together with Claim~\ref{clm:2}, every $i \in [s-1]$ satisfies 
  \begin{equation}\label{eq:double_lowerbound}
    2|X_i| + \sum_{k \neq i}|Y_{i,k}| 
    \ge 2n - 2(s-1)\alpha n - (s-2) \big(\tfrac{2n}{s-1} + \alpha n\big) = 
    \tfrac{2n}{s-1} - (3s-4) \alpha n \ge \tfrac{2n}{s-1} - 3 s \alpha n.
  \end{equation}
  W.l.o.g., let $|Y_{1,2}| = \max_{1 \le i < j \le s-1} |Y_{i,j}|$.
  If $|Y_{1,2}| = 0$, then~\eqref{eq:double_lowerbound}
  guarantees 
  $|X_i| \ge n/(s-1) - \beta' n/2$ for every $i \in [s-1]$, which 
  with Claim~\ref{clm:simpfacts} contradicts Claim~\ref{clm1}.
  Let $Y_{1,2} \neq \emptyset$.
  Lemma~\ref{lem:multigraphturan} and Claim~\ref{clm:simpfacts}
guarantee $e(M[Y_{1,2}]) \le |Y_{1,2}|^2/2$,
  so there exists $v \in Y_{1,2}$ with $d_M(v, Y_{1,2}) \le |Y_{1,2}|$.
  Then, with Claim~\ref{clm:simpfacts}
  \begin{equation*}
    \Big(|X_1| + |X_2| + |Y_{1,2}| +
    \sum_{k \notin \{1,2\}}\big(|Y_{1,k}| + |Y_{2,k}|\big)\Big)
    \le 2n - d_M(v) \le 2n/(s-1) + \alpha n.
  \end{equation*}
  With \eqref{eq:double_lowerbound}, this yields
  \begin{multline*}
    |X_1| + |X_2| + |Y_{1,2}| = \\
    \Big(2|X_1| + \sum_{k \neq 1}|Y_{1,k}|\Big)  +
    \Big(2|X_2| + \sum_{k \neq 2}|Y_{2,k}|\Big) -  
    \Big(|X_1| + |X_2| + |Y_{1,2}| + \sum_{k \notin \{1, 2\}}\big(|Y_{1,k}| + |Y_{2,k}|\big)\Big) \\
    > 2n/(s-1) - \beta' n,
  \end{multline*}
  which contradicts Claim~\ref{clm:3}.

\section{Proof of Lemma~\ref{lem:extremal_rainbow_Kr}}  
Let $(G, c)$ be an $V$-vertex (with $|V| = n$) edge-colored graph 
  satisfying that $\delta^c(G) \ge (1 - 1/(s-1))n$
  and that $(G,c)$ is rainbow $(K_2 \vee K^\ell_{s-2})$-free.
  Observation~\ref{obs:proper_to_rainbow2} guarantees $L \in \mathbb{N}$ so that 
  whenever 
  $A \, \dot\cup \, B \subseteq V(G)$ satisfies that  
  $G[A]$ has an edge, 
  $G[B]$ contains a rainbow $K^L_{s-2}$, and
  $G[A, B] = K[A, B]$ is properly colored by $c$,
  then $G$ contains a contains a rainbow $K_2 \vee K^\ell_{s-2}$.
  Let $\beta, \sigma, \gamma > 0$ and $n \in \mathbb{N}$ satisfy 
  \begin{equation*}
    1/n \ll \beta \ll \sigma \ll \gamma \ll 1/s, 1/L.
  \end{equation*}
  Since $G$ is $(K_s, \beta)$-extremal, there exist pairwise disjoint 
  $V_1, \dotsc, V_{s-1} \subseteq V$ whose every class $V_i$ satisfies 
  $|V_i| \ge n/(s-1) - \beta n$ and $e(G[V_i]) \le \beta n^2$. 
  Observation~\ref{obs:GtoDoutdeg} guarantees that 
  $D = D(G, c, n-1)$ satisfies that 
  \begin{equation}\label{eq:min_out_degree}
    \forall v \in V, \quad d^+_D(v) = d^c(v) \ge (1 - 1/(s-1))n. 
  \end{equation}
  Observation~\ref{obs:prop_coloring} guarantees that $H = H(D)$ is properly colored by $c$.
  Let 
  \[
    W := \left\{v \in V : d^-_D(v) < \left(1 - 1/(s-1) - 0.1 \sigma \right) n\right\}
    \qquad \text{and} \qquad 
  U := V \setminus W.  
  \]

  \begin{claim}\label{clm:upper_bound_on_W}
    $|W| < \sigma^2 n$.
  \end{claim}
  \begin{proof}
    Let $V^- := \{v \in V : d^-_D(v) \ge d^+_D(v)\}$ and 
    $V^+ := V \setminus V^-$.
    Note that 
    \begin{equation*}
      \sum_{v \in V^-} d^-_D(v) - d^+_D(v) = \sum_{v \in V^+} d^+_D(v) - d^-_D(v),
    \end{equation*}
    because $\{V^-, V^+\}$ is a partition of $V$ and 
    $\sum_{v \in V} d^-_D(v) = e(D) = \sum_{v \in V} d^+_D(v)$.
    Recall that for every $i \in [s-1]$, 
    we have $e(G[V_i]) \le \beta n^2$ and $|V_i| \ge (1/(s-1) - \beta)n$. 
    Therefore, 
    with \eqref{eq:min_out_degree}, for every $v \in V$,
    \begin{equation*}
      d^-_D(v) = d^-_D(v, \overline{V_i}) + d^-_D(v, V_i) \le |\overline{V_i}| + d^-_D(v, V_i)
      \le (d^+_D(v) + \beta n) + d^-_D(v, V_i),
    \end{equation*}
    so $d^-_D(v) - d^+_D(v) \le \beta n + d^-_D(v, V_i)$ and
    \begin{equation*}
      \sum_{v \in V^-} d^-_D(v) - d^+_D(v) 
      \le \sum_{i = 1}^{s-1} \left(|V_i \cap V^-| \cdot \beta n + e(G[V_i]\right)
      \le \beta n^2 + (s-1)\beta n^2 = s \beta n^2.
    \end{equation*}
    By \eqref{eq:min_out_degree} and the definition of $W$, for every $v \in W$, 
    \begin{equation*}
      d^+_D(v) - d^-_D(v) \ge \delta^+(D) - d^-_D(v) > 0.1 \sigma n,
    \end{equation*}
    so $W \subseteq V^+$ and 
    \begin{equation*}
      |W| \cdot 0.1 \sigma n < \sum_{v \in V^+} d^+_D(v) - d^-_D(v) 
      = \sum_{v \in V^-} d^-_D(v) - d^+_D(v) \le s \beta n^2. 
    \end{equation*}
    Therefore, $|W| < (s \beta n^2)/(0.1 \sigma n) \le \sigma^2 n$.
  \end{proof}

  By \eqref{eq:min_out_degree}, Claim~\ref{clm:upper_bound_on_W}, and the definitions, for every $v \in U$, we have
  \begin{equation}\label{eq:U_semi_degree_lower_bound}
    d^+_D(v, U), \, d^-_D(v, U) \ge \big(1 - \tfrac{1}{s-1} - 0.1 \sigma\big)n - \sigma^2 n \ge \big(1 - \tfrac{1}{s-1} - 0.2 \sigma \big)n,
  \end{equation}
  and
  \begin{equation}\label{eq:U_double_degree_lower_bound}
    d_H(v, U) = d^+_D(v, U) + d^-_D(v, U) - d_G(v, U) \ge 
    2\big(1 - \tfrac{1}{s-1} - 0.2 \sigma \big)n - n \ge
    \big(1 - \tfrac{2}{s-1} - \sigma \big) n.
  \end{equation}
  For every $i \in [s-1]$, let
  \[
    U'_i := \{v \in V_i \cap U : d_G(v, V_i) \le 0.1 \sigma n\}.
  \]
  We have 
  \begin{equation*}
    |V_i \setminus U'_i| \cdot 0.1 \sigma n \le \sum_{v \in V_i} d_G(v, V_i) = 2 e(G[V_i]) \le 2 \beta n^2,
  \end{equation*}
  so, with Claim~\ref{clm:upper_bound_on_W},
  \begin{equation}\label{eq:U_i_lower_bound}
    |U'_i| \ge |V_i| - |W| - (2 \beta n^2)/(0.1 \sigma n) \ge (1/(s-1) - 2 \sigma^2)n.
  \end{equation}
  Since $U'_i \subseteq V_i$, for every $i \in [s-1]$, we have
  \[    \text{$d_G(v, U'_i) \le 0.1 \sigma n$ for every $v \in U'_i$.}
  \]
  Therefore, define $U_1, \dotsc, U_{s-1}$ 
  to be a collection of pairwise disjoint subsets of $U$ such that 
  \begin{itemize}
    \item for every $i \in [s-1]$, we have $U'_i \subseteq U_i$ and for
      every $v \in U_i$, we have $d_G(v, U_i) \le 0.1 \sigma n$, and
    \item $|U_1 \cup \dotsm \cup U_{s-1}|$ is maximized.
  \end{itemize}
  For every $i \in [s-1]$ and $v\in U_i$, with \eqref{eq:U_semi_degree_lower_bound} and \eqref{eq:U_i_lower_bound} and the fact that $U'_i \subseteq U_i$, 
  we have
  \begin{equation}\label{eq:U_i_H_deg_lower_bound}
    d_H(v, \overline{U_i}) = d^+_D(v, \overline{U_i}) + d^-_D(v, \overline{U_i}) - d_G(v, \overline{U_i}) \ge 
    (|\overline{U_i}| - 0.2 \sigma n) + (|\overline{U_i}| - 0.3 \sigma n) - |\overline{U_i}|
    \ge |\overline{U_i}| - \sigma n.
  \end{equation}

  \begin{claim}\label{clm:double_degree_in_r-1_parts} 
    For every $v \in V$, if there exists $i \in [s-1]$ such that for every $j \in [s-1] \setminus \{i\}$
    we have $d_H(v, U_j) \ge \gamma n$, then $d_G(v, U_i) = 0$.
  \end{claim}
  \begin{proof}
    For a contradiction, assume that there exists $i \in [s-1]$ such that there exists $u \in N_G(v, U_i)$
    and $d_H(v, U_j) \ge \gamma n$ for every $j \in [s-1] \setminus \{ i \}$.
    Let $A := \{u,v\}$.
    Since $L s \cdot \sigma n \le \gamma n \le d_H(v, U_j)$ for every $j \in [s-1] \setminus \{ i \}$, 
    \eqref{eq:U_i_H_deg_lower_bound} implies that we can form a set $B$ by, for every $j \in [s-1] \setminus \{i\}$, adding $L$ vertices from $U_j$
    so that $H[A, B]$ is a complete bipartite graph and $H[B]$ is isomorphic to $K_{r-2}^L$.
    This means that both $G[B]$ and $G[A, B]$ are properly colored by $c$, a contradiction.
  \end{proof}

  Note that \eqref{eq:U_i_H_deg_lower_bound}, \eqref{eq:U_i_lower_bound}, 
  and Claim~\ref{clm:double_degree_in_r-1_parts} imply that $e(G[U_i]) = 0$ for every $i \in [s-1]$.
  Let 
  \[
    U_0 := U \setminus (U_1 \cup \dotsm \cup U_{s-1}).
  \]
  By the maximality of $|U_1 \cup \dotsm \cup U_{s-1}|$, we have
  \begin{equation}\label{eq:U_0neighU_i}
    \text{$d_G(v, U_i) > 0$ for every $v \in U_0$ and $i \in [s-1]$.}
  \end{equation}

  \begin{claim}\label{clm:U_0_is_empty}
    $U_0 = \emptyset$.
  \end{claim}
  \begin{proof}
    Suppose, for a contradiction, that there exists $v \in U_0$.
    By \eqref{eq:U_0neighU_i} and Claim~\ref{clm:double_degree_in_r-1_parts},
    there are distinct $i,j \in [s-1]$ such that $d_H(v, U_i), d_H(v, U_j) < \gamma n$.
    Let $U_{\overline{ij}} := U \setminus (U_i \cup U_j)$.  With \eqref{eq:U_i_lower_bound}
    we have 
    \begin{equation*}
      |U_{\overline{ij}}| \le |V| - |U_i| - |U_j| \le (1 - 2/(s-1) + 4\sigma^2)n,
    \end{equation*}
    so with \eqref{eq:U_double_degree_lower_bound} we have $d_H(v, U) \ge |U_{\overline{ij}}| - 2\sigma n$.
    This with our selection of $i$ and $j$ gives us
    \begin{equation}\label{eq:lower_bound_on_v_into_Uprime}
      d_H(v, U_{\overline{ij}}) = d_H(v, U) - d_H(v, U_i \cup U_j) \ge (|U_{\overline{ij}}| - 2 \sigma n) - 2 \gamma n \ge |U_{\overline{ij}}| - 3 \gamma n, 
    \end{equation}
    and with \eqref{eq:U_semi_degree_lower_bound}, we also have 
    \begin{equation}\label{eq:lower_bound_on_v_into_U}
      d_G(v, U) = d^+_D(v, U) + d^-_D(v, U) - d_H(v, U) \ge 2(1 - 1/(s-1) - 0.2\sigma)n - (|U_{\overline{ij}}| + 2 \gamma n) \ge (1 - 3 \gamma)n.
    \end{equation}
    Without loss of generality we can assume that $d^+_D(v, U_i) \ge d^+_D(v, U_j)$. 
    Then with \eqref{eq:U_semi_degree_lower_bound},
    \begin{equation*}
      d^+_D(v, U_i) \ge (d^+_D(v, U_i) + d^+_D(v, U_j))/2 \ge (d^+_D(v, U) - |U_{\overline{ij}}|)/2 \ge 
      (1/(s-1) - 2 \sigma)n/2 \ge L,
    \end{equation*}
    so there exists $B' \subseteq N^+_D(v, U_i)$ such that $|B'| \ge L$.
    By \eqref{eq:U_i_H_deg_lower_bound} and \eqref{eq:lower_bound_on_v_into_U}, 
    there exists $A' \subseteq N_G(v, U_j)$ such that $|A'| = L+1$
    and $H[A',B']$ is a complete bipartite graph. 
    For every $y \in B'$, because $A' \subseteq N^+_D(y)$, 
    there exists at most one $u \in A'$ such that $c(uy) = c(vy)$.
    Therefore, because $|A'| - |B'| \ge 1$, there exists 
    $u \in A'$ such that if $A := \{u,v\}$, then for every $y \in B'$, we have $c(uy) \neq c(vy)$.
    This together with the fact that $B' \subseteq N^+_D(u) \cap N^+_D(v)$ imply 
    that the complete bipartite graph $G[A, B']$ is properly colored.
    By \eqref{eq:U_i_H_deg_lower_bound} and \eqref{eq:lower_bound_on_v_into_Uprime}, 
    we can extend $B'$ to $B$ by adding $L$ vertices
    from $V_k$ for every $k \in [s-1] \setminus \{i, j\}$ so that $H[B]$ is isomorphic to $K_{s-2}^L$,
    and so that for every $x \in A$ and $y \in B \setminus B'$, 
    we have $A \subseteq N^+_D(y)$ and $B \subseteq N^+_D(x)$. 
    This means that both $G[B]$ and $G[A, B]$ are properly colored by $c$, a contradiction.
  \end{proof}

  \begin{claim}\label{clm:U_is_K_r-free}
    $G[U]$ is $K_s$-free.
  \end{claim}
  \begin{proof} 
    Claims~\ref{clm:U_0_is_empty} implies 
    that $\{U_1, \dotsc, U_{s-1}\}$ is a partition of $U$.
    Since $e(G[U_i]) = 0$ for $i \in [s-1]$, the graph $G[U]$ is $K_s$-free.
  \end{proof}

  \begin{claim}\label{clm:W_is_empty}
    $W = \emptyset$. 
  \end{claim}
  \begin{proof}
    By Claim~\ref{clm:upper_bound_on_W}, we have 
    $|U| \ge n - |W| \ge (1 - \sigma^2)n$.
    With Claim~\ref{clm:U_is_K_r-free} and Tur\'an's Theorem, we have $e(G[U]) \le (1 - 1/(s-1)) |U|^2/2$.
    Assume for a contradiction that $W \neq \emptyset$.
    Then,
    \begin{multline*}
      e^+_D(U, W) = \sum_{v \in U}d^+_D(v) - d^+_D(v, U) \ge (1 - 1/(s-1))n \cdot |U| - 2 \cdot e(G[U]) \ge \\
      (1 - 1/(s-1))n \cdot |U| - (1 - 1/(s-1)) |U|^2 = (1 - 1/(s-1))|W| |U|,
    \end{multline*}
    and there exists $v \in W$ such that 
    \begin{equation*}
      d^-_D(v, U) \ge (1 - 1/(s-1)) |U| \ge (1 - 1/(s-1)) (1 - \sigma^2)n \ge (1 - 1/(s-1) - \sigma^2)n.
    \end{equation*}
    This is a contradiction, because, by the definition of $W$, $d^-_D(v, U) \le d^-_D(v) \le (1 - 1/(s-1) - 0.1 \sigma)n$.
  \end{proof}

  Claims~\ref{clm:U_is_K_r-free} and \ref{clm:W_is_empty} imply that $G$ is $K_r$-free.
  The result then follows from Tur\'an's Theorem and the fact that 
  $2 e(G) \ge n \cdot \delta(G) \ge n \cdot \delta^c(G) \ge (1 - 1/(s-1))n^2$.

\end{document}